\documentclass[12pt,leqno]{article}

\usepackage{amsmath,amssymb,amsthm}
\usepackage[all]{xy}
\usepackage{lscape}

\makeatletter

\newtheorem{theorem}{Theorem}[section]
\newtheorem{proposition}[theorem]{Proposition}
\newtheorem{lemma}[theorem]{Lemma}
\newtheorem{corollary}[theorem]{Corollary}

\theoremstyle{definition}

\newtheorem{example}[theorem]{Example}
\newtheorem{definition}[theorem]{Definition}

\theoremstyle{remark}

\theoremstyle{remark}
\newtheorem{remark}[theorem]{Remark}

\def\({{\rm (}}
\def\){{\rm )}}

\let\Mathrm\operator@font
\let\Cal\mathcal
\let\Bbb\mathbb
\let\bs\boldsymbol

\def\standop#1{\mathop{\Mathrm #1}\nolimits}
\def\difstop#1#2{\expandafter\def\csname #1\endcsname{\standop{#2}}}
\def\defstop#1{\difstop{#1}{#1}}

\defstop{AB}
\defstop{ann}
\defstop{Ass}

\defstop{CMFI}
\defstop{codim}
\defstop{Coh}
\defstop{coht}
\defstop{Coker}
\defstop{Cone}
\defstop{Cl}
\defstop{Cox}
\defstop{cosk}
\defstop{card}

\defstop{depth}
\defstop{Div}
\defstop{div}

\defstop{EM}
\defstop{embdim}
\defstop{End}
\defstop{ev}
\defstop{Ext}

\defstop{Flat}
\defstop{Func}
\defstop{Fpqc}
\defstop{fl}

\defstop{GD}
\defstop{Good}

\difstop{height}{ht}
\defstop{Hom}

\def\id{\mathord{\Mathrm{id}}}
\def\Id{\mathord{\Mathrm{Id}}}
\difstop{Image}{Im}
\defstop{ind}
\defstop{Inv}

\defstop{Ker}

\defstop{Lch}
\defstop{length}
\defstop{Lin}
\defstop{Lqc}
\defstop{lqc}
\defstop{LQ}
\defstop{LM}
\defstop{Lie}

\def\mon{^\mathrm{mon}}
\defstop{Mat}
\defstop{Max}
\defstop{Min}
\defstop{Mod}
\defstop{Mor}
\defstop{MCM}
\defstop{Map}

\defstop{Nerve}
\defstop{NonSFR}
\defstop{NonCMFI}
\defstop{NonNor}
\defstop{Nor}

\defstop{ob}
\defstop{Ob}
\def\op{^{\standop{op}}}

\defstop{PA}
\defstop{Ps}
\defstop{PM}
\defstop{Proj}
\defstop{Prin}
\defstop{Pic}

\defstop{Qch}
\defstop{qch}
\defstop{Qfpqc}

\defstop{rad}
\defstop{rank}

\defstop{res}
\defstop{Reg}
\defstop{Ref}

\def\Sch{\underline{\Mathrm Sch}}

\defstop{Spec}
\defstop{supp}
\defstop{Supp}
\defstop{Sym}
\defstop{Sing}
\defstop{SFR}
\defstop{Soc}

\defstop{Sh}

\difstop{tdeg}{trans.deg}

\defstop{Tor}
\difstop{trace}{tr}

\defstop{Zar}

\def\cc{\mathop{\check H}\nolimits}
\def\ccs{\mathop{\check{\underline H}}\nolimits}

\def\km{\kappa\text{-}\Mor}

\def\C{\Cal C}

\def\F{\Cal F}

\def\L{\Cal L}
\def\M{\Cal M}
\def\N{\Cal N}
\def\O{\Cal O}

\let\indlim\varinjlim

\def\sdarrow#1{\downarrow\hbox to 0pt{\scriptsize$#1$\hss}}
\def\suarrow#1{\uparrow\hbox to 0pt{\scriptsize$#1$\hss}}
\def\ssearrow#1{\searrow\hbox to 0pt{\scriptsize$#1$\hss}}

\def\section{\@startsection{section}{1}{\z@ }%
  {-3.5ex plus -1ex minus -.2ex}{2.3ex plus .2ex}{\bf }}

\long\def\refname{\par\kern -3ex
  \begin{center}\rm R\sc{eferences}\end{center}\par\kern 
  -2ex}

\def\@seccntformat#1{\csname the#1\endcsname.\quad}

\def\@@@sect#1#2#3#4#5#6[#7]#8{%
  \ifnum #2>\c@secnumdepth 
  \def \@svsec {}\else \refstepcounter {#1}%
  \def\@svsec{}
  \fi 
  \@tempskipa #5\relax 
  \ifdim \@tempskipa >\z@ 
  \begingroup #6\relax \@hangfrom {\hskip #3\relax 
    \@svsec}{\interlinepenalty \@M #8\par }\endgroup 
  \csname #1mark\endcsname {#7}
  \else 
  \def \@svsechd {#6\hskip #3\@svsec #8\csname #1mark\endcsname {#7}}
  \fi \@xsect {#5}}

\def\@@@startsection#1#2#3#4#5#6{%
  \if@noskipsec \leavevmode \fi \par \@tempskipa #4\relax \@afterindenttrue 
  \ifdim \@tempskipa <\z@ \@tempskipa -\@tempskipa \@afterindentfalse 
  \fi \if@nobreak \everypar {}\else \addpenalty {\@secpenalty }\addvspace 
  {\@tempskipa }\fi \@ifstar {\@ssect {#3}{#4}{#5}{#6}}{\@dblarg 
    {\@@@sect {#1}{#2}{#3}{#4}{#5}{#6}}}}

\def\theparagraph{\thesection.\arabic{paragraph}}
\def\aparagraph{\@@@startsection{paragraph}{2}{\z@ }%
  {1.75ex plus .2ex minus .15ex}{-1em}{\bf(\theparagraph) } }
\def\paragraph{\@@@startsection{paragraph}{2}{\z@ }%
  {1.75ex plus .2ex minus .15ex}{-1em}{}{\bf(\theparagraph)} }

\let\c@theorem\c@paragraph

\title{Equivariant class group. II. \\
Enriched descent theorem}

\author{M{\sc itsuyasu} H{\sc ashimoto}}
\date{\normalsize
  Department of Mathematics, Okayama University\\
  Okayama 700--8530, JAPAN\\
  {\small \tt mh@okayama-u.ac.jp}}

\begin{document}

\maketitle

\begin{abstract}
We prove a version of Grothendieck's descent theorem
on an `enriched' principal fiber bundle, a principal fiber bundle with 
an action of a larger group scheme.
Using this, we prove the isomorphisms of the equivariant Picard and the
class groups arising from such a principal fiber bundle.
\end{abstract}

\section{Introduction}

This paper is a continuation of \cite{Hashimoto4}.

Let $S$ be a scheme, and $G$ a flat $S$-group scheme.
In \cite{Hashimoto4}, we have defined the equivariant class group 
$\Cl(G,X)$ for a locally Krull $G$-scheme $X$.
Utilizing Grothendieck's descent theorem, we have proved that
for a principal $G$-bundle $\varphi:X\rightarrow Y$, 
$Y$ is locally Krull, and the inverse image functor $\varphi^*$ 
induces an isomorphism $\Cl(Y)\rightarrow \Cl(G,X)$ 
(\cite[(5.32)]{Hashimoto4}).

In this paper, we generalize this to an enriched version.
Let $f:G\rightarrow H$ be an fpqc homomorphism, and $N:=\Ker f$.
Let $\varphi:X\rightarrow Y$ be a $G$-morphism which is also
a principal $N$-bundle.
We call such a morphism a {\em $G$-enriched principal $N$-bundle}
(the name \lq equivariant principal bundle' is reserved for a 
different notion, which is deeply related, see 
Example~\ref{equivariant-bundle.ex}).
The purpose of this paper is to prove the enriched descent theorem,
which yields an isomorphism 
$\varphi^{*}:\Cl(H,Y)\rightarrow \Cl(G,X)$.
A similar isomorphism
$\varphi^{*}:\Pic(H,Y)\rightarrow\Pic(G,X)$ is also proved.

The isomorphism is induced by the corresponding equivalence
of the categories of quasi-coherent sheaves, 
$\varphi^*:\Qch(H,Y)\rightarrow \Qch(G,X)$.
This enriched Grothendieck's descent is conceptually trivial,
and probably checked relatively easily for simpler cases.
In this paper, we do not assume that $\varphi$ is flat, quasi-compact, 
quasi-separated, or locally of finite type
(these assumptions are automatically satisfied if $N\rightarrow S$ 
satisfies the same conditions, see Lemma~\ref{principal-descent.thm}).
However, the original Grothendieck's descent is known under
the same general settings (even more is known, see \cite[(4.46)]{Vistoli}),
and we try to prove the enriched version without these redundant assumptions.

The quasi-inverse of $\varphi^*$ is $(?)^N\circ \varphi_*$, 
the direct image followed by the $N$-invariance.
However, this does not mean $\varphi_*$ preserves the quasi-coherence,
as we do not assume that $\varphi$ is quasi-compact quasi-separated.
As we do not assume that $\varphi$ is flat, there are some technical
problems in treating the quasi-coherent sheaves on the small Zariski site,
and it is more comfortable to treat the big site with flat topology.
Finally, as 
we do not put any finiteness assumptions on $\varphi:X\rightarrow Y$,
it is suitable to treat the fpqc topology, not the fppf topology.

As a biproduct, we have an isomorphism
$\varphi^{*}:\Pic(H,Y)\rightarrow \Pic(G,X)$ for non-flat huge groups
(Corollary~\ref{princ-pic-equiv.thm}).

Sections~\ref{qfpqc.sec} to \ref{gro.sec} are devoted to preliminaries.

In section~\ref{qfpqc.sec}, we define {\em quasi-fpqc} ({\em qfpqc} for short) 
morphisms of schemes (Definition~\ref{qfpqc.def}).
The topology defined by quasi-fpqc morphisms is the same as that defined
by fpqc morphisms.
However, any group scheme is qfpqc, and it is comfortable to treat
qfpqc morphisms when we do not put the flatness assumption on groups.
We point out that various properties of morphisms of schemes descend
with respect to qfpqc morphisms (Lemma~\ref{qfpqc-descent.thm}).

In section~\ref{kappa.sec}, we give a way to make the big fpqc site
skeletally small in a reasonable way.
We measure the \lq size' of a morphism by a cardinal, and we only treat
the morphisms whose size is bounded by a fixed regular cardinal.
This way, we can avoid set-theoretic problems in treating fpqc topology.
Sizing schemes using a cardinal is treated in \cite[(3.9)]{SP},
and our approach can be viewed as a relative version of that in \cite{SP}.
Our definition of $\kappa$-morphisms for a regular cardinal $\kappa$
enables us to treat all schemes $T$ (of arbitrary size) as a base
scheme (see the condition {\bf 4} in (\ref{cond-C.par})).

Section~\ref{sheaves.sec} is a preliminary on sheaves on ringed sites.
As an abstraction of Kempf's result \cite[Theroem~8]{Kempf}, we give a 
sufficient condition for cohomology functors to be compatible with
direct limits (Lemma~\ref{Kempf.thm}).
Using this, we compare quasi-coherent sheaves on different ringed sites.

In section~\ref{gro.sec}, we introduce several sites related to 
diagrams of schemes.
It is convenient to grasp $G$-equivariant quasi-coherent sheaves on $X$
as a quasi-coherent sheaves over the finite diagram of schemes 
$B_G^M(X)$, and a well-behaved treatment of dualizing complexes is
known for this category \cite{ETI}.
However, sometimes we want to consider the full simplicial schemes.
It seems that 
this would be necessary when we pursue the cohomological descent as in
\cite{Saint-Donat}.
Moreover, for our purpose, considering big sites is necessary.
Some of our argument in section~\ref{e-gro.sec} 
does not work for small Zariski site.
The point is the exactness of the inverse image (Lemma~\ref{inverse-image.thm})
and the fact that Lipman's theta is an isomorphism 
(Lemma~\ref{theta.thm}).

On the other hand, the big sites with non-flat morphisms are
too big when we consider the derived category, since the 
full subcategory of quasi-coherent sheaves is not closed under
kernels in the whole category of modules.
So we compare these sites, and give a proof for the fact that the
categories of quasi-coherent sheaves are all equivalent, although
it seems that this is well-known for experts.

In the last section, section~\ref{e-gro.sec}, we prove the enriched
Grothendieck's descent theorem.
Our main theorem is on the equivariant modules on the
fpqc site, Theorem~\ref{principal-main.thm}.
The author does not know if the same statement for the small Zariski site
is true.
We get the corresponding assertion for quasi-coherent sheaves 
(note that the quasi-coherent sheaves are essentially 
independent of the choice of the ringed site, as proved in 
section~\ref{gro.sec}) immediately (Corollary~\ref{descent-main-cor.thm}).
We do not know how to prove Corollary~\ref{descent-main-cor.thm} directly
without using the fpqc site in this generality, although it seems that
the additional assumption that $G$ is flat and quasi-compact quasi-separated
would make it possible relatively easily.
The choice of the fpqc site is effective in our proof.
See the proof of Lemma~\ref{complicated.thm} and 
Theorem~\ref{principal-main.thm}.

As corollaries, we prove the isomorphisms between the equivariant 
Picard and the class groups arising from an enriched principal
fiber bundle, see Corollary~\ref{princ-pic-equiv.thm} and 
Corollary~\ref{princ-ref-equiv.thm}.

We will see some applications of the enriched descent in the 
continuation of this paper \cite{Hashimoto5}.

Acknowledgment:
The author is grateful to Professor Shouhei Ma for kindly showing him
the reference \cite{Vistoli}.

\section{Quasi-fpqc morphisms and enriched principal bundles}
\label{qfpqc.sec}

\paragraph
This paper is a continuation of \cite{Hashimoto4}.
We follow the notation and terminology there, unless otherwise specified.
Throughout this paper, let $S$ be a scheme.

\begin{definition}\label{qfpqc.def}
A morphism of schemes $\varphi:X\rightarrow Y$ is said to be {\em quasi-fpqc}
(or {\em qfpqc} for short)
if there exists some morphism $\psi:Z\rightarrow X$ such that
$\varphi\psi$ is fpqc.
\end{definition}

\begin{lemma}\label{qfpqc-basic.thm}
The following hold true.
\begin{enumerate}
\item[\bf 1] A morphism $\varphi:X\rightarrow Y$ is fpqc if and only 
if it is qfpqc and flat.
\item[\bf 2] A base change of a qfpqc morphism is qfpqc.
\item[\bf 3] A composite of qfpqc morphisms is qfpqc.
\item[\bf 4] A qfpqc morphism is submersive.
In particular, it is surjective.
\item[\bf 5] A group scheme $G$ over $S$ is qfpqc over $S$.
\end{enumerate}
\end{lemma}

\begin{proof}
{\bf 1} The \lq only if' part is trivial.
We prove the \lq if' part.
So there is a morphism $\psi:Z\rightarrow X$ such that $\varphi\psi$ is
fpqc.
Let $U$ be a quasi-compact open subset of $Y$.
Then there is a quasi-compact open subset $V$ of $Z$ such that
$\varphi(\psi(V))=U$.
As $\psi(V)$ is quasi-compact and contained in $\varphi^{-1}(U)$, 
there is a quasi-compact open subset $W$ of $\varphi^{-1}(U)$ containing
$\psi(V)$.
Then $\varphi(W)=U$.
So $\varphi$ is fpqc.

{\bf 2} and {\bf 3} are easy.
{\bf 4} follows from \cite[(2.3.12)]{EGA-IV-2}.
{\bf 5} is because $S\rightarrow G\rightarrow S$ is the identity,
where the first map $S\rightarrow G$ is the unit of the group scheme.
\end{proof}

\begin{lemma}\label{qfpqc-descent.thm}
Let $\Bbb P$ be a property of morphisms of schemes.
Assume that
\begin{enumerate}
\item[\bf 1] \(Zariski local property\)
If $f:X\rightarrow Y$ is a morphism and $Y=\bigcup_i U_i$ is an 
affine open covering, then 
$f$ satisfies $\Bbb P$ if and only if
$f|_{f^{-1}(U_i)}:f^{-1}(U_i)\rightarrow U_i$
satisfies $\Bbb P$ for each $i$.
\item[\bf 2] Let $f:X\rightarrow Y$ be a morphism,
$g:Y'\rightarrow Y$ a 
morphism, and assume that $Y$ and $Y'$ are affine.
\begin{enumerate}
\item[\rm(i)] \(base change for affine bases\)
If $f$ satisfies $\Bbb P$, then the base change
$f':Y'\times_Y X\rightarrow Y'$ of $f$ by $g$ satisfies $\Bbb P$.
\item[\rm(ii)] \(flat descent for affine bases\)
If $f'$ satisfies $\Bbb P$ and $g$ is faithfully flat, then $f$ satisfies
$\Bbb P$.
\end{enumerate}
\end{enumerate}
Then if $f:X\rightarrow Y$ is a morphism of schemes, $g:Y'\rightarrow Y$
a qfpqc morphism of schemes and the base change $f'$ satisfies $\Bbb P$, 
then $f$ satisfies $\Bbb P$.
In particular, if $f'$ satisfies one of 
separated, quasi-compact, quasi-separated, 
locally of finite presentation, proper, affine, finite, flat, faithfully
flat, smooth, 
unramified, \'etale, submersive, 
a closed immersion, an open immersion, an immersion, 
an isomorphism, fpqc, and qfpqc, then $f$ satisfies the
same property.
If $\Bbb Q$ is a property of an algebra essentially of finite type over a 
field such that the base change and the descent holds for any base 
field extension, and $\Bbb P$ is \lq any fiber satisfies $\Bbb Q$,' then
$\Bbb P$ descends with respect to a qfpqc base change.
So if $\Bbb Q$ is one of geometrically normal, geometrically reduced, 
Cohen--Macaulay, Gorenstein, and local complete intersection, and $f'$
satisfies $\Bbb P$, then $f$ satisfies $\Bbb P$.
\end{lemma} 

\begin{proof}
The first assertion is rather formal, and is left to the reader.
Most of the examples of $\Bbb P$ are listed in \cite[(2.36)]{Vistoli}.
We only prove that the property $\Bbb P=\text{qfpqc}$ satisfies
{\bf 2}, (ii).
As $f'$ is qfpqc, there is a morphism $h:U\rightarrow X'$ such that 
$f'h$ is fpqc.
As $g$ is also fpqc, the composite 
$gf'h=fg'h$ is fpqc, where $g':X'=Y'\times_Y X
\rightarrow X$ is the second projection.
So $f$ is qfpqc, as required.
\end{proof}

\paragraph
From now on, unless otherwise specified,
let $G$ be an $S$-group scheme,
and $N\subset G$ a normal subgroup scheme of $G$.
That is, $N$ is a subscheme of $G$, $N$ itself is an $S$-group scheme,
and the inclusion $N\hookrightarrow G$ is a homomorphism such that
$G\times N\rightarrow G$ $((g,n)\mapsto gng^{-1})$ factors through $N$.

\begin{definition}
  We say that $\varphi:X\rightarrow Y$ is a {\em $G$-enriched principal
  $N$-bundle} if it is a $G$-morphism, and is a principal $N$-bundle
\cite[(2.6)]{Hashimoto4}.
\end{definition}

The following is immediate from the definition.

\begin{lemma}\label{principal-base-change.thm}
  Let $\varphi:X\rightarrow Y$ be a $G$-enriched principal $N$-bundle,
  and $h:Y'\rightarrow Y$ be a $G$-morphism.
  If the action of $N$ on $Y'$ is trivial, then
  the second projection $p_2: X\times_YY'\rightarrow Y'$ is a 
  $G$-enriched principal $N$-bundle.
  \qed
\end{lemma}

\begin{lemma}\label{principal-mumford.thm}
An $N$-invariant $G$-morphism $\varphi:X\rightarrow Y$ is a $G$-enriched
principal $N$-bundle if and only if $\varphi$ is qfpqc and the
map $\Phi:N\times X\rightarrow X\times_Y X$ given by 
$\Phi(n,x)=(nx,x)$ is an isomorphism.
\end{lemma}

\begin{proof}
This is immediate from \cite[(4.43)]{Vistoli}.
\end{proof}

\begin{lemma}\label{fpqc-equal.thm}
A qfpqc morphism is an epimorphism.
That is,   
if $f:X\rightarrow Y$ is a qfpqc morphism,
  $g$ and $h$ are morphisms $Y\rightarrow Z$, and
$gf=hf$, then $g=h$.
\end{lemma}

\begin{proof}
  Set $S:=\Spec \Bbb Z$.
  Then there is a commutative diagram
  \[
  \xymatrix{
    X_0 \ar[d]^{d_X} \ar[r] & Y_0 \ar[d]^{d_Y} \ar[r] & Z \ar[d]^\Delta \\
    X \ar[r]^f & Y \ar[r]^-{(g,h)} & Z\times_S Z
  }
  \]
  with cartesian squares.
  By assumption, $d_X$ is an isomorphism.
  By Lemma~\ref{qfpqc-descent.thm}, 
$d_Y$ is an isomorphism, and hence $g=h$.
\end{proof}

\begin{corollary}\label{trivial.thm}
  Let $S$ be a scheme, $F$ an $S$-group scheme, and $\varphi:X\rightarrow Y$ 
an 
  qfpqc $F$-morphism.
  If the action of $F$ on $X$ is trivial, then the action of $F$ on $Y$ is
  trivial.
\end{corollary}

\begin{proof}
  Let $f=1_F\times \varphi: F\times X\rightarrow F\times Y$.
  It is qfpqc.
  Then $af=p_2f$, as $\varphi$ is an $F$-morphism and the action of $F$ on
  $X$ is trivial, where $a:F\times Y\rightarrow Y$ is the action, and
  $p_2:F\times Y\rightarrow Y$ is the second projection.
  By Lemma~\ref{fpqc-equal.thm}, $a=p_2$, and the action of $F$ on $Y$
  is trivial.
\end{proof}

\begin{lemma}\label{principal-locally-trivial.thm}
  Let $X$ and $Y$ be $G$-schemes, and $\varphi:X\rightarrow Y$ be a
  $G$-morphism.
  Then the following are equivalent.
  \begin{enumerate}
  \item[\bf 1] $\varphi$ is a $G$-enriched principal $N$-bundle.
  \item[\bf 2] 
    There exists some qfpqc $G$-morphism
    $h:Y'\rightarrow Y$ such that the second projection 
    $p_2:X\times_Y Y'\rightarrow Y'$ is a $G$-morphism
which is a trivial $N$-bundle.
  \item[\bf 3] 
    There exists some qfpqc $S$-morphism
    $h:Y'\rightarrow Y$ such that the second projection 
    $p_2:X'=X\times_Y Y'\rightarrow Y'$ is a principal $N$-bundle.
  \end{enumerate}
\end{lemma}

\begin{proof}
  {\bf 1$\Rightarrow$2} Let $h=\varphi$, and use
Lemma~\ref{principal-mumford.thm}.

{\bf 2$\Rightarrow$3} is trivial.

{\bf 3$\Rightarrow$1} There exists some fpqc $h':Y''\rightarrow Y'$ such that
the base change $X''\rightarrow Y''$ is a trivial $N$-bundle.
As $hh'$ is qfpqc, there exists some $h'':Y'''\rightarrow Y''$ such that
$hh'h''$ is fpqc.
As the base change $X'''\rightarrow Y'''$ is also a trivial $N$-bundle,
$\varphi:X\rightarrow Y$ is a principal $N$-bundle.
As $\varphi$ is assumed to be a $G$-morphism, it is a $G$-enriched
principal $N$-bundle.
\end{proof}

\begin{lemma}\label{principal-descent.thm}
  Let $\varphi:X\rightarrow Y$ be a principal $G$-bundle,
and $\Bbb P$ a property of morphisms of schemes as in
Lemma~\ref{qfpqc-descent.thm}.
  If $G\rightarrow S$ satisfies $\Bbb P$, then 
$\varphi$ also satisfies $\Bbb P$.
\end{lemma}

\begin{proof}
There is an fpqc morphism $Y'\rightarrow Y$ such that the base change
$\varphi':Y'\times_Y X\rightarrow Y'$ is a trivial $G$-bundle.
  By the Zariski local property and the base change for affine bases,
it is easy to check the general base change property, and hence
$\varphi'$ satisfies $\Bbb P$.
By Lemma~\ref{qfpqc-descent.thm}, $\varphi$ satisfies $\Bbb P$.
\end{proof}

\section{$\kappa$-schemes}
\label{kappa.sec}

\paragraph Let $\kappa$ be a cardinal.
A topological space $X$ is said to be {\em $\kappa$-compact} 
if every open cover of 
$X$ has a subcover of cardinality strictly less than $\kappa$.
So $X$ is $1$-compact (resp.\ $2$-compact, $\aleph_0$-compact, 
$\aleph_1$-compact) if and only if $X$ is empty (resp.\ local 
\cite[(8.7)]{HO}, (quasi-)compact, 
Lindel\"of), where $\aleph_0=\#\Bbb Q$, and $\aleph_1$ is the successor 
cardinal of $\aleph_0$.

If $f:X\rightarrow Y$ is a surjective continuous map and $X$ is
$\kappa$-compact, then so is $Y$.

\paragraph
A cardinal is said to be regular if it is equal to its own cofinality
\cite[p.~257]{Enderton}.
In our paper, a regular cardinal is required to be infinite.
So $\kappa$ is regular if and only if for any decomposition
$\kappa=\bigcup_{i\in I}S_i$ for a family of subsets with $\#I<\kappa$, 
$\#S_i=\kappa$ for at least one $i$ \cite[Theorem~9T]{Enderton}.
Note that $\aleph_0$ and infinite successor cardinals are regular.
By the definition, an inaccessible cardinal is regular \cite[p.~254]{Enderton}.

Let $\kappa$ be regular.
Let $(Y_i)_{i\in I}$ be a family of subspaces of
$X$.
If $\bigcup_i Y_i=X$, each $Y_i$ is $\kappa$-compact, and 
$\#I<\kappa$, then $X$ is also $\kappa$-compact.

\paragraph
Let $\kappa$ be infinite.
(The underlying space of) a scheme is $\kappa$-compact if and only if
it is covered by an affine open covering of cardinality strictly less than 
$\kappa$.

\paragraph
A morphism of schemes $\varphi:X\rightarrow Y$ is said to be 
$\kappa$-compact if there exists some affine open covering $(U_i)_{i\in I}$ of 
$Y$ such that each $\varphi^{-1}(U_i)$ is $\kappa$-compact.
If $\kappa<\kappa'$ and $\varphi$ is $\kappa$-compact, then it is 
$\kappa'$-compact.
A quasi-compact morphism is nothing but an $\aleph_0$-compact morphism.
So a quasi-compact morphism is $\kappa$-compact.

\paragraph
Note that
$\varphi:X\rightarrow Y$ is $\kappa$-compact, and
$h:Y'\rightarrow Y$ is another morphism such that $Y'$ is quasi-compact,
then $Y'\times_Y X$ is $\kappa$-compact.
In particular, a base change of a $\kappa$-compact morphism is again 
$\kappa$-compact.
If $Y$ is quasi-compact, $\varphi:X\rightarrow Y$ is $\kappa$-compact, 
then $X$ is $\kappa$-compact.

\paragraph
Let 
$\kappa$ be regular.
If $\varphi:X\rightarrow Y$ is a morphism of schemes, 
$Y$ is $\kappa$-compact, and $\varphi$ 
is $\kappa$-compact, then $X$ is $\kappa$-compact.
It follows that a composite of $\kappa$-compact morphisms is again 
$\kappa$-compact in this case.

\paragraph
Let $\kappa$ be a regular cardinal.
A morphism of schemes $\varphi:X\rightarrow Y$ is said to be
$\kappa$-quasi-separated if the diagonal $\Delta_{X/Y}:X\rightarrow 
X\times_Y X$ is $\kappa$-compact.
A quasi-separated morphism is nothing but an $\aleph_0$-quasi-separated
morphism.
As in \cite[(1.2)]{EGA-IV-1}, we can prove the following.
An immersion is $\kappa$-quasi-separated.
A base change of a $\kappa$-quasi-separated morphism is 
$\kappa$-quasi-separated.
Let $f:X\rightarrow Y$ and $g:Y\rightarrow Z$ be morphisms.
If $f$ and $g$ are $\kappa$-quasi-separated, then so is $gf$.
If $gf$ is $\kappa$-quasi-separated, then so is $f$.
If $g$ is $\kappa$-quasi-separated and $gf$ is $\kappa$-compact, then
$f$ is $\kappa$-compact.

\paragraph
We call a $\kappa$-compact $\kappa$-quasi-separated morphism a 
$\kappa$-concentrated morphism.
The composition of $\kappa$-concentrated morphisms is $\kappa$-concentrated.
The base change of a $\kappa$-concentrated morphism is $\kappa$-concentrated.
If $gf$ and $g$ are $\kappa$-concentrated, then so is $f$.

\paragraph
Let $\lambda$ be an infinite cardinal.
Let $h:A\rightarrow B$ be a map of commutative rings.
We say that $h$ is of $\lambda$-type if $B$ is generated by a subset
whose cardinal is strictly less than $\lambda$ over $A$.
So $h$ is of finite type if and only if $h$ is of $\aleph_0$-type.
We say that a ring $A$ is $\lambda$-type if $\Bbb Z\rightarrow A$ is
of $\lambda$-type.
If $\lambda>\aleph_0$, then $A$ is of $\lambda$-type if and only if 
$\#A<\lambda$.
In particular, if $\lambda>\aleph_0$, then a subring of a ring of 
$\lambda$-type is again of $\lambda$-type.
If $A$ is a pure subring of $B$, and $B$ is a ring of $\lambda$-type,
then $A$ is also of $\lambda$-type.
If $\lambda>\aleph_0$, then this is because $A\subset B$.
If $\lambda=\aleph_0$, then $B$ is of finite type over $\Bbb Z$, and
hence so is $A$ by \cite{Hashimoto6}.
In particular, if $A\rightarrow B$ is faithfully flat and $B$ is
of $\lambda$-type, then so is $A$.
Let $B$ be an $A$-algebra, and $C$ a $B$-algebra.
If $B$ is of $\lambda$-type over $A$ and $C$ is of $\lambda$-type
over $B$, then $C$ is of $\lambda$-type over $A$.
If $C$ is of $\lambda$-type over $A$, then it is of $\lambda$-type
over $B$.

\paragraph
A morphism of schemes $\varphi:X\rightarrow Y$ is said to be 
locally of $\lambda$-type if there exists some 
affine open covering $(U_i)_{i\in I}$ and affine open covering
$(V^i_j)_{j\in J_i}$ of $\varphi^{-1}(U_i)$ for each $i$ such that 
$\Gamma(U_i,\O_Y)\rightarrow \Gamma(V^i_j,\O_X)$ is of $\lambda$-type.
A morphism locally of finite-type is locally of $\lambda$-type.
It is easy to see that a base change of a locally $\lambda$-type morphism
is again locally of $\lambda$-type.
Note that $\Spec B\rightarrow \Spec A$ is locally of $\lambda$-type if
and only if $A\rightarrow B$ is of $\lambda$-type.
The composite of two locally $\lambda$-type morphisms is again locally of
$\lambda$-type.
In particular, a morphism of schemes $\varphi:X\rightarrow Y$ is
locally of $\lambda$-type if and only if for any affine open subscheme 
$U=\Spec A$ of $Y$ and any affine open subscheme $V=\Spec B$ of $\varphi^{-1}
(U)$, $A\rightarrow B$ is of $\lambda$-type.
If $f:X\rightarrow Y$ and $g:Y\rightarrow Z$ are morphisms and
$gf$ is locally of $\lambda$-type, then so is $f$.

\paragraph
From now on, until the end of this paper, let $\kappa$ denote a
regular cardinal.
We say that a morphism of schemes $\varphi:X\rightarrow Y$ is
a $\kappa$-morphism if it is locally of $\kappa$-type and $\kappa$-concentrated.
A scheme $X$ is said to be a $\kappa$-scheme if $X\rightarrow\Spec \Bbb Z$ 
is a $\kappa$-morphism.
The composite of two $\kappa$-morphisms is a $\kappa$-morphism.
In particular, if $Y$ is a $\kappa$-scheme and $\varphi:X\rightarrow Y$ 
is a $\kappa$-scheme, then $X$ is a $\kappa$-scheme.
A base change of a $\kappa$-morphism is a $\kappa$-morphism.
In particular, a direct product of $\kappa$-schemes is a $\kappa$-scheme.
If $f:X\rightarrow Y$ and $g:Y\rightarrow Z$ are morphisms and
$gf$ is a $\kappa$-morphism, then so is $f$.

\paragraph
For a ring $A$, the number of quasi-compact open subsets of $\Spec A$ is
finite or less than or equal to $\#A$, because the 
compact open subsets of the form $\Spec A[1/f]$ form an open basis.
It follows that if a ring $A$ is of $\kappa$-type, then any subset 
of $\Spec A$ is $\kappa$-compact.
So any subscheme of a $\kappa$-scheme is again a $\kappa$-scheme.
If $X$ is a $\kappa$-scheme, then a local ring $\O_{X,x}$ is a $\kappa$-ring.
If $\varphi:X\rightarrow Y$ is an fpqc morphism and $X$ is a 
$\kappa$-scheme, then $Y$ is a $\kappa$-scheme.
For a given set of schemes $\Omega$, there exists some $\kappa$
such that any element of $\Omega$ is a $\kappa$-scheme.

\paragraph
We can add the properties of morphisms $\kappa$-compact, 
$\kappa$-quasi-separated, locally of $\kappa$-type to the list of
properties $\Bbb P$ in Lemma~\ref{qfpqc-descent.thm}.
So these properties descends with respect to qfpqc base change.
In particular, by  Lemma~\ref{basic.thm}, {\bf 5} and
Lemma~\ref{principal-descent.thm}, a principal $G$-bundle is a 
$\kappa$-morphism, if $G\rightarrow S$ is so.

\paragraph For a given scheme $S$, 
we denote the full subcategory of $\Sch/S$ consisting of
objects such that the structure morphisms are $\kappa$-morphisms
by $(\Sch/S)_\kappa$.
If $T\in (\Sch/S)_\kappa$, then $(\Sch/S)_\kappa/T$ is the same as
$(\Sch/T)_\kappa$.

\begin{lemma}
Let $S$ be a scheme, and $\kappa$ a regular cardinal.
Then the category $(\Sch/S)_\kappa$ is skeletally small.
\end{lemma}

\begin{proof}
Replacing $\kappa$ if necessary, we may assume that $S$ is a $\kappa$-scheme.
So $(\Sch/S)_\kappa$ is equivalent to $(\Sch/\Bbb Z)_\kappa/S$.
If a category $\Cal C$ is skeletally small and $c\in\Cal C$, then 
$\Cal C/c$ is skeletally small.
So We may assume that $S=\Spec \Bbb Z$.
Let $\Cal R$ be a complete set of representatives of isomorphism classes of
$\Bbb Z$-algebras of $\kappa$-type.
Note that $\Cal R$ can be a subset of the set of quotients of the
polynomial ring
$\Bbb Z[x_\alpha]_{\alpha\in\kappa}$, and certainly a small set.
Thus the category $\Cal A$ of affine $\kappa$-schemes is skeletally small.

Let $I_0=\kappa$ and $I_1=I_0\times I_0\times\kappa$.
Let $I$ be the small category with the object set $\Ob(I)=I_0\coprod I_1$,
and $\Hom(\bs i,\bs j)$ is a singleton if $\bs j=(j,j',k)\in I_1$ and $\bs i
\in \{j,j'\}\subset I_0$, and empty otherwise.
The category of $I\op$-diagrams of affine 
$\kappa$-schemes $\Cal B=\Func(I\op,\Cal A)$ is skeletally small.
Any $\kappa$-scheme can be expressed as 
$X=\bigcup_{\bs i\in I_0}U_i$ with $U_i\in \Cal A$.
Note that each $U_i\cap U_j$ is $\kappa$-compact, since $X$ is 
$\kappa$-quasi-separated.
So for each $(i,j)$, there is a covering $U_i\cap U_j=\bigcup_{k\in\kappa}U_{ijk}$
with $U_{ijk}\in\Cal A$.
So the collections determines an object $\Cal X=(((U_i),(U_{ikj})))$ 
of $\Cal B$, and $X$ is the colimit of $\Cal X$.
By the uniqueness of the colimit, the category of $\kappa$-schemes is
also skeletally small.
\end{proof}

\section{Module sheaves over a ringed site}
\label{sheaves.sec}

\paragraph\label{qc-invertible.par}
Let $\Cal C=(\Cal C,\Cal O_{\Cal C})$ be a ringed site.
That is, $\Cal C$ is a site (a category equipped with a pretopology)
and $\Cal O_{\Cal C}$ is a sheaf of commutative rings on $\Cal C$.

The category of $\O_{\Cal C}$-modules is denoted by $\Mod(\Cal C)$.
A free sheaf on $\Cal C$ is an $\O_{\Cal C}$-module which is isomorphic
to a direct sum of copies of the $\O_{\Cal C}$-module $\O_{\Cal C}$.
A sheaf of $\Cal O_{\Cal C}$-modules $\M$ is said to be quasi-coherent
(resp.\ invertible)
if for any $c\in\Cal C$, there exists some covering $(c_\lambda\rightarrow
c)$ such that for each $\lambda$, there is an exact sequence of sheaves of 
$\O_{\Cal C}|_{c_{\lambda}}$-modules of the form
\[
\Cal F_1 \rightarrow \Cal F_0 \rightarrow \M|_{c_\lambda}\rightarrow 0
\]
with $\Cal F_1$ and $\Cal F_0$ free
(resp.\ $\Cal F_1=0$ and $\Cal F_0=\O_{\Cal C}|_{c_\lambda}$).
Obviously, a quasi-coherent sheaf is invertible.
The category of quasi-coherent sheaves (resp.\ invertible sheaves) on 
$\Cal C$ is denoted by $\Qch(\Cal C)$ (resp.\ $\Inv(\Cal C)$).

\paragraph
Assume that $\Cal C$ has an initial object $X$.
Then for an $\Gamma(X,\O_{\Cal C})$-module $M$, A presheaf $\tilde M^p$ is
defined by $\Gamma(c,\tilde M^p)=\Gamma(c,\O_{\Cal C})\otimes_{
\Gamma(X,\O_{\Cal C})}M$.
Its sheafification is denoted by $\tilde M$.

An $\O_{\Cal C}$-module of the form $\tilde M$ is quasi-coherent.

\paragraph\label{inverse-image.par}
Let $\Cal C$ and $\Cal D$ be sites and $f:\Cal D\rightarrow\Cal C$ a
functor.
In this paper, we say that $f$ is {\em continuous} if for any $d\in\Cal D$ and
any covering $(d_\lambda\rightarrow d)$, $(fd_\lambda\rightarrow fd)$ is
a covering again.
Let $f=(f,\eta):(\Cal C,\Cal O_{\Cal C})\rightarrow(\Cal D,\Cal O_{\Cal D})$ 
be a morphism of ringed sites.
That is, $f:\Cal D\rightarrow \Cal C$ is a continuous functor,
and $\eta:\Cal O_{\Cal D}\rightarrow f_*\Cal O_{\Cal C}$ is a map of 
sheaves of commutative rings.
Then almost by definition, if $\Cal M\in\Mod(\Cal D)$ is quasi-coherent
(resp.\  invertible), then so is $f^*\Cal M$.

\paragraph
Let $\Cal C$ be a site.
An $E_{\Cal C}$-morphism is a morphism that appears as a part of a covering.
A subset $B$ of $\Ob(\Cal C)$ is said to be a basis of the topology of
$\Cal C$ if for any $U\in\Ob(\Cal C)$ and any covering 
$(V_\lambda\rightarrow U)$, there is a refinement $(W_\mu\rightarrow U)$ of
$(V_\lambda\rightarrow U)$ such that each $W_\mu$ belongs to $B$.
An object $U$ of $\Cal C$ is said to be quasi-compact if
for any covering $(V_\lambda\rightarrow U)_{\lambda\in\Lambda}$ 
of $U$, there is a finite subset $\Lambda_0$ of $\Lambda$ such that
$(V_\lambda\rightarrow U)_{\Lambda_0}$ is also a covering of $U$.
An object $U$ of $\Cal C$ is said to be quasi-separated if
for any quasi-compact objects $V$, $W$ and any 
$E_{\Cal C}$-morphisms $V\rightarrow U$
and $W\rightarrow U$, $V\times_U W$ is again quasi-compact.
A subset $B$ of $\Ob(\Cal C)$ is said to be quasi-compact (resp.\ 
quasi-separated) if each element of $B$ is quasi-compact (resp.\ 
quasi-separated).

We say that $\Cal C$ is locally concentrated if $\Cal C$ has a
quasi-compact quasi-separated base of topology.

\paragraph
Let $\Cal C$ be a site with a basis of topology $B$.
Let $\Cal D$ be a full subcategory of $\Cal C$ such that $\Ob(\Cal D)$
consists of quasi-compact quasi-separated objects.
Assume that $\Ob(\Cal D)\supset B$.

We say that a presheaf of abelian groups 
$\M$ on $\Cal D$ is a $\Cal D$-sheaf, if 
for any $U\in\Cal D$, 
for any finite covering $(V_i\rightarrow U)$ of $U$ with $V_i\in B$, for any
coverings $(V_{ijk}\rightarrow V_i\times_U V_j)$ of $V_i\times_U V_j$ with
$V_{ijk}\in B$, the sequence
\[
0\rightarrow \Gamma(U,\M)\rightarrow \prod_i \Gamma(V_i,\M)
\rightarrow \prod_{i,j}\prod_k \Gamma(V_{ijk},\M)
\]
is exact.
Let $\Ps(\Cal D)$ and $\Ps(\Cal C)$ be the category of 
presheaves of abelian groups, and let $\Sh(\Cal C)$ be the category of 
sheaves of abelian groups.
Let $\Sh(\Cal D)$ be the category of $\Cal D$-sheaves.

Let $\iota:\Cal D\rightarrow \Cal C$ be the inclusion.
Then there is an obvious commutative diagram
\[
\xymatrix{
\Sh(\Cal C) \ar[r]^{q_{\Cal C}} \ar[d]^{\iota^\#_{\Sh}} &
\Ps(\Cal C) \ar[d]^{\iota^\#_{\Ps}} \\
\Sh(\Cal D) \ar[r]^{q_{\Cal D}} &
\Ps(\Cal D)
},
\]
where $q$ are the inclusions.

\begin{lemma}\label{sheaf-extension.thm}
Let the notation be as above.
Then there is a functor $\gamma:\Ps(\Cal D)\rightarrow \Sh(\Cal C)$ such that
the following are satisfied.
\begin{enumerate}
\item[\bf 1] $\gamma\iota^\#:\Ps(\Cal C)\rightarrow \Sh(\Cal C)$ 
is a sheafification, that is, left adjoint to $q$.
\item[\bf 2] $\iota^\#\gamma:\Ps(\Cal D)\rightarrow \Sh(\Cal D)$
is a sheafification, that is, left adjoint to $q$.
\end{enumerate}
Moreover, we have
\begin{enumerate}
\item[\bf 3] 
$\iota^\#_{\Sh}:\Sh(\Cal C)\rightarrow \Sh(\Cal D)$ is an equivalence
with $\gamma q$ its quasi-inverse.
\item[\bf 4] If $\Cal L\in\Ps(\Cal C)$ and $\iota^\#_{\Ps}(\Cal L)$ is
a $\Cal D$-sheaf, then the canonical map
$\Gamma(U,\Cal L)\rightarrow \Gamma(U,{}^a\Cal L)$ is an isomorphism
for $U\in\Cal D$.
\end{enumerate}
\end{lemma}

\begin{proof}
Let $\M\in\Ps(\Cal D)$.

For each object $W$ of $\Cal C$, define $\cc^0(W,\M)$ to be the
inductive limit of 
\[
\Ker(\prod_{i\in I} \Gamma(V_i,\M)\rightarrow \prod_{ij}\prod_{k\in J_{ij}}
\Gamma(V_{ijk},\M)),
\]
where $(V_i\rightarrow W)_{i\in I}$ runs through the coverings of $W$ with
$V_i\in B$ (the index set $I$ may be infinite), and
$(V_{ijk}\rightarrow V_i\times_W V_j)$ runs through the coverings of 
$V_i\times_W V_j$ with $V_{ijk}\in B$.
It is easy to see that $\ccs^0(\M)=\cc^0(?,\M)$ is a separated 
presheaf of abelian groups on $\Cal C$, and $\ccs^0(\M)|_{\Cal D}$ is
$\M$.
Now as can be seen easily, $\N=\ccs^0(\ccs^0(\M))$ is a sheaf of abelian groups,
and $\N|_{\Cal D}=\M$.
This defines a functor $\gamma:=
\ccs^0\ccs^0:\Ps(\Cal D)\rightarrow \Sh(\Cal C)$.
By construction, {\bf 1} and {\bf 2} are satisfied.

The other assertions are proved easily.
\end{proof}

\begin{lemma}[cf.~{\cite[Theorem~8]{Kempf}}]\label{Kempf.thm}
Let $\Cal C$ be a locally concentrated site, and $U$ a 
quasi-compact quasi-separated object of $\Cal C$.
Let $\Cal F_\alpha$ be a filtered inductive system of abelian sheaves on
$\Cal C$.
Then the canonical map
\[
\indlim H^i(U,\Cal F_\alpha)\rightarrow H^i(U,\indlim \Cal F_\alpha)
\]
is an isomorphism.
\end{lemma}

\begin{proof}
First consider the case that $i=0$.
Let $L$ be the inductive limit of $(\F_\alpha)$ in the category of 
presheaves so that the left hand side is $\L(U)$.

Let $\Cal V=(V_i\rightarrow U)$ be any finite 
covering with $V_i$ quasi-compact quasi-separated,
and 
$(V_{ijk}\rightarrow V_i\times_U V_j)$ be finite coverings 
with $V_{ijk}$ quasi-compact quasi-separated.
Now
\[
0\rightarrow \F_\alpha(U)\rightarrow \prod_i \F_\alpha(V_i)\rightarrow
\prod_{i,j,k}\F_\alpha(V_{ijk})
\]
is exact, since each $F_\alpha$ is a sheaf.
Taking the inductive limit, 
\[
0\rightarrow \L(U)\rightarrow \prod_i \L(V_i)\rightarrow
\prod_{i,j,k}\L(V_{ijk})
\]
is also exact, since the inductive limit is compatible with any 
finite product.
Applying Lemma~\ref{sheaf-extension.thm}, {\bf 4} 
for the full subcategory $\Cal D$
of $\Cal C$ consisting of all the quasi-compact quasi-separated objects,
we have the case for $i=0$ immediately, as
the sheafification ${}^a\L$ of $\L$ is $\indlim\F_\alpha$.

Now consider the general case.
Let $I^\bullet_\alpha$ be the injective resolution of $F_\alpha$.
It suffices to show that $\indlim I^\bullet_\alpha$ is an 
$\Gamma(U,?)$-acyclic resolution of $\indlim F_\alpha$.
To show this, it suffices to show that if $(I_\alpha)$ is an inductive 
system of abelian sheaves which are 
$\Gamma(U,?)$-acyclic for any quasi-compact quasi-separated $U$, then 
$H^i(U,\indlim I_\alpha)=0$ for any quasi-compact quasi-separated 
$U$ and $i>0$.
This is checked as in \cite[(2.12)]{Milne}, and is left to the reader.
\end{proof}

\begin{corollary}\label{locally-tilde.thm}
Let $(\C,\O)$ be a locally concentrated ringed site.
Then for any quasi-coherent sheaf $\M$ over $\C$ and $c\in\C$, 
there exists some covering $(c_\lambda\rightarrow c)$ such that 
for each $\lambda$, $\M|_{c_\lambda}$ is isomorphic to $\tilde M_\lambda$ for
some $\Gamma(c_\lambda,\O)$-module $M_\lambda$.
\end{corollary}

\begin{proof}
There is a covering $(c_\lambda\rightarrow c)$ such that for each $\lambda$, 
there exists some exact sequence of the form
\[
\Cal F_{\lambda,1}\xrightarrow{d_\lambda}\Cal F_{\lambda,0}
\rightarrow\Cal M|_{c_\lambda}\rightarrow 0
\]
with $\Cal F_{\lambda,1}$ and $\Cal F_{\lambda,0}$ free.
We may assume that each $c_\lambda$ is quasi-compact quasi-separated.

Let $M_\lambda$ be the cokernel of the map
\[
F_1:=\Gamma(c_\lambda,\Cal F_{\lambda,1})\xrightarrow{d_\lambda}
\Gamma(c_\lambda,\Cal F_{\lambda,0})=:F_0.
\]
Then there is a commutative diagram
\[
\xymatrix{
\tilde F_1 \ar[d]^{h_1} \ar[r]^{d_\lambda} &
\tilde F_0 \ar[d]^{h_0} \ar[r] &
\tilde M \ar[d]^h \ar[r] & 0 \\
\Cal F_{\lambda,1} \ar[r]^{d_\lambda} &
\Cal F_{\lambda,0} \ar[r] &
\Cal M|_{c_\lambda} \ar[r] & 0
},
\]
where $h_1$, $h_0$ and $h$ are canonical maps.
Thus it suffices to show that $h_1$ and $h_0$ are isomorphisms.
This is proved easily,
using Lemma~\ref{Kempf.thm}.
\end{proof}

\section{Grothendieck's descent}
\label{gro.sec}

\paragraph\label{cond-C.par}
Let $\Bbb C$ be a class of morphisms of schemes.
We assume that
\begin{enumerate}
\item[\bf 1] All isomorphisms are in $\Bbb C$.
\item[\bf 2] If $f:X\rightarrow Y$ and $g:Y\rightarrow Z$ are morphisms
and $g\in\Bbb C$, then $f\in\Bbb C$ if and only if $gf\in\Bbb C$.
\item[\bf 3] If $f:X\rightarrow Y$ is in $\Bbb C$ and $Y'\rightarrow Y$ 
any morphism of schemes, then the base change $f':X'=Y'\times_Y X\rightarrow Y'
$ of $f$ is again in $\Bbb C$.
\end{enumerate}
Sometimes we identify $\Bbb C$ as a subcategory of $\Sch$ such that
$\Ob(\Bbb C)=\Ob(\Sch)$ and $\Mor(\Bbb C)=\Bbb C$.
When we take the fiber product 
of two morphisms with the same codomain in $\Bbb C$ 
in the category $\Sch$, it is also a fiber product in $\Bbb C$.
For a scheme $T$, the category $\Bbb C/T$ is a full subcategory of $\Sch/T$.
Further, we assume that 
\begin{enumerate}
\item[\bf 4] $\Bbb C/T$ is skeletally small for any scheme $T$.
\end{enumerate}

\paragraph\label{cond-E.par}
Let $E$ be a Grothendieck pretopology on 
$\Sch/S$ such as Zariski, \'etale, fppf, and fpqc.
An $E$-morphism is a morphism in $\Sch/S$ which appear as a morphism 
of some covering 
(such as an open immersion, \'etale map,
flat map locally of finite presentation, and flat map).
We assume that an $E$-morphism is in $\Bbb C$.
We also assume that the presheaf $\O$ of rings 
defined by $\O(X)=\Gamma(X,\O_X)$ 
is a sheaf with respect to $E$.

\paragraph \label{diagram.par}
Let $I$ be a small category, and
$X_\bullet$ an $I\op$-diagram of $S$-schemes.

We define $(\Bbb C/X_\bullet)_E$ to be the ringed site defined as follows.
An object of $(\Bbb C/X_\bullet)_E$ is a pair $(i,f:U\rightarrow X_i)$ 
such that $i\in \Ob(I)$ and $f\in\Bbb C$.
A morphism from $(j,f:U\rightarrow X_j)$ to $(i,g:V\rightarrow X_i)$ is
a pair $(\phi,h)$ such that $\phi:i\rightarrow j$ is a morphism in
$I$, and $h:U\rightarrow V$ is a morphism such that $gh= X_\phi f$.
The composition is given by $(\phi',h')\circ(\phi,h)=(\phi\phi',h'h)$.
Thus $(\Bbb C/X_\bullet)_E$ is a skeletally small category.
A family $((\phi_\lambda,h_\lambda):(i_\lambda,U_\lambda)\rightarrow (i,U))$ 
with the same codomain $(i,U)$ is a covering if $i_\lambda=i$,
$\phi_\lambda=\id_\lambda$, and the collection
$(h_\lambda:U_\lambda\rightarrow U)$ is a covering of $U$ in $E$.
The structure sheaf $\O$ is defined in an appropriate way.

\paragraph
The first example is the case that $\Bbb C=(\Zar)$ 
is the class of open immersions
and $E=\Zar$ is the Zariski topology.
The resulting site $((\Zar)/X_\bullet)_{\Zar}$ is called the small Zariski site
and is denoted by $\Zar(X_\bullet)$.
We denote $\Mod(\Zar(X_\bullet))$, $\Qch(\Zar(X_\bullet))$, 
and $\Inv(\Zar(X_\bullet))$ by $\Mod(X_\bullet)$,
$\Qch(X_\bullet)$, and $\Inv((X_\bullet)$, respectively.

\paragraph
Let $\kappa$ be a regular cardinal.
Let $\Bbb C=(\km)$ be the class of all $\kappa$-morphisms.
The resulting site 
$((\km)/X_\bullet)_{\Zar}$ is called the 
$\kappa$-big Zariski site,
and is denoted by $\Zar^b(X_\bullet)=\Zar^b_\kappa(X_\bullet)$.

We denote
$\Mod(\Zar^b_\kappa(X_\bullet))$, $\Qch(\Zar^b_\kappa(X_\bullet))$, 
and $\Inv(\Zar^b_\kappa(X_\bullet))$ by $\Mod^b_\kappa X_\bullet)$,
$\Qch_\kappa^b(X_\bullet)$, and $\Inv_\kappa^b(X_\bullet)$, respectively.
The suffix $\kappa$ will be omitted if there is no danger of confusion.

\paragraph\label{Zar-description.par}
There is an obvious inclusion $\mu:\Zar(X_\bullet)\rightarrow
\Zar(X_\bullet)_\kappa^{b}$.
It gives a morphism of ringed sites
$\mu:\Zar(X_\bullet)_\kappa^b\rightarrow \Zar(X_\bullet)$.
It is easy to see that for each $\M\in\Mod(\Zar(X_\bullet))$ and 
a $\kappa$-morphism $(i,f:U\rightarrow X_i)$, $\Gamma((i,f),
\mu^*(\M))=\Gamma(U,f^*(\M_i))$.

\paragraph
Let $E=\Fpqc$ (resp.\ $\Qfpqc$)
be the pretopology whose covering $(V_\lambda\rightarrow U)$ is
a family of $S$-morphisms such that
the induced morphism $\coprod_\lambda V_\lambda\rightarrow U$ is fpqc
(resp.\ qfpqc).
Almost by definition, $\Qfpqc$ is a saturation of $\Fpqc$.
In what follows, we do not treat $\Qfpqc$, but implicitly use the fact that a 
sheaf for $\Fpqc$ is also a sheaf for $\Qfpqc$
\cite[(2.48)]{Vistoli}.

We denote the site $((\km)/X_\bullet)_{\Fpqc}$
by $\Fpqc(X_\bullet)=\Fpqc_\kappa(X_\bullet)$.
By 
\cite[(2.60), (4.22)]{Vistoli}, the sheaf $\Cal O$
of commutative rings is a sheaf on $\Fpqc(X_\bullet)$,  and thus
$\Fpqc(X_\bullet)$ is a ringed site.
We denote $\Mod(\Fpqc(X_\bullet))$, $\Qch(\Fpqc(X_\bullet))$, and 
$\Inv(\Fpqc(X_\bullet))$ by 
$\Mod_{\Fpqc}(X_\bullet)$, $\Qch_{\Fpqc}(X_\bullet)$,
and $\Inv_{\Fpqc}(X_\bullet)$, respectively.

We have a morphism of ringed sites
$\nu:\Fpqc(X_\bullet) \rightarrow\Zar^b(X_\bullet)$.

\paragraph
Note that $(i,f:U\rightarrow X_i)$ with $U$ affine forms a quasi-compact
quasi-separated basis of topology of $\Zar(X_\bullet)$.
Similarly for $\Zar^b(X_\bullet)$ and $\Fpqc(X_\bullet)$.
Thus these sites are locally concentrated.

\paragraph
Considering the case that $I=1$ (the discrete 
category with one object), we have $\Zar(X)$, 
$\Zar^b(X)$, and $\Fpqc(X)$ for a single scheme $X$.

\paragraph
Until the end of this section, let $\Bbb C$ and $E$ be as in
(\ref{cond-C.par}) and (\ref{cond-E.par}).
Let $I$ and $X_\bullet$ be as in (\ref{diagram.par}).

\begin{lemma}\label{restriction.thm}
The restriction functor $(?)_J:\Mod((\Bbb C/X_\bullet)_E)\rightarrow 
\Mod((\Bbb C/X_\bullet)_E)$ has both a left adjoint and a right adjoint.
In particular, $(?)_J$ preserves arbitrary limits and colimits.
In particular, $(?)_J$ is exact.
\end{lemma}

\begin{proof}
The left adjoint exists by Kan's lemma (see e.g., 
\cite[Theorem~I.2.1]{Artin}.
The existence of the right adjoint is by \cite[(2.31)]{ETI}.
\end{proof}

\begin{lemma}\label{inverse-image-theta.thm}
Let $I$ be a small category, and $f_\bullet:X_\bullet\rightarrow Y_\bullet$ be
a morphism of $I\op$-diagrams of schemes.
Let $J$ be a subcategory of $I$.
Then Lipman's theta {\rm\cite[(3.7.2)]{Lipman}}, {\rm\cite[(1.21)]{ETI}}
\[
\theta:(f_\bullet)_J^*(?)_J\rightarrow (?)_Jf_\bullet^*
\]
is an isomorphism of functors $\Mod((\Bbb C/Y_\bullet)_E)
\rightarrow \Mod((\Bbb C/X_\bullet|_J)_E)$.
\end{lemma}

\begin{proof}
This is proved similarly to \cite[(6.25)]{ETI}.
\end{proof}

\begin{lemma}\label{inverse-image.thm}
Let $I$ be a small category, and $f_\bullet:X_\bullet\rightarrow Y_\bullet$ be
a morphism of $I\op$-diagrams of schemes.
If $f_i\in\Bbb C$ for any $i\in I$, then $f_\bullet^*:\Mod((\Bbb C/Y_\bullet)_E)
\rightarrow \Mod((\Bbb C/X_\bullet)_E)$ is nothing but the restriction.
That is, 
for $\M\in\Mod((\Bbb C/Y_\bullet)_E)$ and $(i,g:U\rightarrow X_i)
\in \Bbb C/X_\bullet$,
$\Gamma((i,g),f_\bullet^*\M)=\Gamma((i,f_ig),\M)$.
In particular, $f^*_\bullet$ is an exact functor.
\end{lemma}

\begin{proof} By Lemma~\ref{inverse-image-theta.thm} applied to
$J=\{i\}$, we may 
assume that $I=1$, and the problem is on a morphism $f:X\rightarrow Y$
of single schemes which lies in $\Bbb C$.
Then the inverse image of the presheaf module $(f^*)^p\M$ 
is given by 
$\Gamma(U,(f^*)^p\M)=\indlim \Gamma(U,\O_X)\otimes_{\Gamma(V,\O_Y)}\Gamma(V,\M)$,
where the colimit is taken over pairs $(V,h)$
with $V\in \Bbb C/Y$ and $h\in \Mor(U,V)$.
But as $f\in\Bbb C$, the colimit is taken over a category with the 
final object $(U,1_U)$, and hence $\Gamma(U,(f^*)^p\M)=\Gamma(U,\M)$.
So $(f^*)^p\M$ is already a sheaf, and the assertion follows.
\end{proof}

\begin{lemma}\label{theta.thm}
Let $I$ be a small category, and 
\[
\xymatrix{
X_\bullet' \ar[d]^{f_\bullet'} \ar[r]^{g^X_\bullet} &
X_\bullet \ar[d]^{f_\bullet} \\
Y_\bullet' \ar[r]^{g^Y_\bullet} & Y_\bullet
}
\]
be a cartesian square of $I\op$-diagrams of $S$-schemes.
Assume that $(g^Y_\bullet)_i\in\Bbb C$ for any $i\in I$.
Then Lipman's theta 
\[
\theta: (g^Y_\bullet)^*(f_\bullet)_*\rightarrow (f'_\bullet)_*(g^X_\bullet)^*
\]
between the functors $\Mod((\Bbb C/X_\bullet)_E)\rightarrow 
\Mod((\Bbb C/Y_\bullet')_E)$ is an isomorphism.
\end{lemma}

\begin{proof}
By \cite[(3.7.2)]{Lipman} and Lemma~\ref{inverse-image-theta.thm}, 
it suffices to prove that
\[
\theta:(g^Y_i)^*(f_i)_*\rightarrow (f'_i)_*(g^X_i)^*
\]
is an isomorphism for each $i\in I$, and we may assume that 
$I=1$.
Then for $\M\in\Mod((\Bbb C/X)_E)$ and $U\in \Bbb C/Y'$, 
$\theta$ is the identity map on 
$\Gamma(U\times_Y X,\M)$, and is an isomorphism.
\end{proof}

\begin{lemma}\label{tilde-isom.thm}
Let $(i,f:U\rightarrow X_i)\in\Zar(X_\bullet)$ and $M$ a 
$\Gamma((i,f),\O_{X_\bullet})$-module.
Then for any $(\phi,h):(j,g)\rightarrow(i,f)\in\Zar(X_\bullet)/(i,f)$
with $V=\Spec A$, the source of $g$, affine,
we have that the canonical map
\[
\Gamma((j,g),\tilde M^p)=A\otimes_{\Gamma((i,f),\O_{X_\bullet})}M\rightarrow
\Gamma((j,g),\tilde M)
\]
is an isomorphism.
Similar results hold for the other two sites.
\end{lemma}

\begin{proof}
This follows immediately from \cite[(4.22)]{Vistoli} and
Lemma~\ref{sheaf-extension.thm}.
\end{proof}

\begin{lemma}\label{qch.thm}
Let $\M\in\Mod(X_\bullet)$.
Then the following are equivalent.
\begin{enumerate}
\item[\bf 1] $\M$ is quasi-coherent.
\item[\bf 2] For any 
$(i,f:U\rightarrow X_i)\in 
\Zar(X_\bullet)$ \(resp.\  $\Zar^b(X_\bullet)$, 
$\Fpqc(X_\bullet)$\) with $U=\Spec A$ affine, $\M|_{(i,f)}$ is canonically 
isomorphic to $\tilde M$, where $M=\Gamma((i,f),\M)$.
\item[\bf 3] For any morphism $(\phi,h):(j,g:V\rightarrow X_j)
\rightarrow (i,f:U\rightarrow X_i)$ in 
$\Zar(X_\bullet)$ \(resp.\  $\Zar^b(X_\bullet)$, 
$\Fpqc(X_\bullet)$\) with $U=\Spec A$ and $V=\Spec B$ affine,
the canonical map $B\otimes_A \Gamma((i,f),\M)\rightarrow \Gamma((j,g),\M)$ 
is an isomorphism.
\end{enumerate}
Similar results hold for the other two sites.
\end{lemma}

\begin{proof}
{\bf 1$\Rightarrow$2}.
It suffices to show that
$\tilde M\cong \M|_{(i,f)}$ for some $A$-module $M$.
By Corollary~\ref{locally-tilde.thm}, 
for some faithfully flat algebra $A'$ of $A$, 
$\tilde N\cong \M|_{(i,fh)}$, where $h:\Spec A'\rightarrow \Spec A$ is
the canonical map, and $N=\Gamma((i,fh),\M)$.
Now we can find such an $M$ by the descent theory \cite[(4.21)]{Vistoli}.

{\bf 2$\Rightarrow$3} follows from
Lemma~\ref{tilde-isom.thm}.

{\bf 3$\Rightarrow$1} Let $(i,f:U\rightarrow X_i)$ be an object of
$\Zar(X_\bullet)$ (resp.\ $\Zar^b(X_\bullet)$, $\Fpqc(X_\bullet)$) with
$U=\Spec A$ affine.
Then $\M|_{(i,f)}\cong \tilde M$ for $M=\Gamma((i,f),\M)$ by
the uniqueness assertion in Lemma~\ref{sheaf-extension.thm}.
The assertion follows.
\end{proof}

\begin{lemma}\label{inv.thm}
Let $\varphi:X\rightarrow Y$ be an qfpqc morphism of schemes.
Let $\M$ be an $\O_Y$-module.
Then $\M$ is quasi-coherent \(resp.\ an invertible sheaf\) if and only if
$\varphi^*\M$ is so.
\end{lemma}

\begin{proof}
If $\M$ is quasi-coherent or invertible, then obviously $\varphi^*\M$ is so.
In order to prove the converse, as the question is local on $Y$, we
may assume that $Y$ is affine.
Then we may assume that $X$ is affine and $\varphi$ is flat.
The assertion follows from \cite[(10.14)]{ETI} for quasi-coherence.
To prove the assertion for invertible sheaves, it suffices to show that
if $A\rightarrow B$ is a homomorphims of commutative rings, $M$ an 
$A$-module, and $B\otimes_AM$ is rank-one projective, then $M$ is
rank-one projective.
As $B\otimes_A M$ is finitely presented and flat, $M$ is so.
Thus $M$ is finite projective.
It is easy to see that $M$ has rank one.
\end{proof}

\begin{lemma}
The functors $\mu^*:\Qch(X_\bullet)\rightarrow\Qch^{b}(X_\bullet)$ and
$\mu_*:\Qch^{b}(X_\bullet)\rightarrow\Qch(X_\bullet)$ are well-defined and
quasi-inverse each other.
Similar results hold for the morphism 
of ringed sites $\nu:\Fpqc(X_\bullet)\rightarrow \Zar^b(X_\bullet)$.

These equivalences induce equivalences 
of the categories of invertible sheaves
$\Inv(X_\bullet)$, $\Inv^b(X_\bullet)$, 
and $\Inv_{\Fpqc}(X_\bullet)$.
\end{lemma}

\begin{proof}
Well-definedness of $\mu^*$ is (\ref{inverse-image.par}).
Well-definedness of $\mu_*$ is checked by the condition
{\bf 3} of Lemma~\ref{qch.thm}.

If $(i,f:\Spec A\rightarrow X_i)$ is an object of $\Zar(X_\bullet)$ and
$M$ an $A$-module, then $\mu^*(\tilde M)\cong \tilde M$ and
$\mu_*(\tilde M)\cong \tilde M$.
Thus $\mu^*$ and $\mu_*$ are quasi-inverse each other.
Similarly for $\nu$.

It is obvious that the inverse image preserves invertible sheaves.
As $\mu_*(\tilde M)\cong \tilde M$,
it suffices to show that for any faithfully flat ring homomorphism 
$A\rightarrow B$ and an $A$-module $M$, $M$ is rank-one projective
if and only if $B\otimes_A M$ is rank-one projective,
in order to prove that $\mu_*$ preserves invertible sheaves.
This is done in the proof of Lemma~\ref{inv.thm}.

Similarly for $\nu$.
\end{proof}

\begin{lemma}\label{mor-munu.thm}
For a morphism of $I\op$-diagrams of schemes $f_\bullet:X_\bullet\rightarrow
Y_\bullet$, the diagram
\[
\xymatrix{
\Fpqc(X_\bullet) \ar[r]^\nu \ar[d]^{f_\bullet} &
\Zar^b(X_\bullet) \ar[r]^\mu \ar[d]^{f_\bullet} &
\Zar(X_\bullet) \ar[d]^{f_\bullet} \\
\Fpqc(Y_\bullet) \ar[r]^\nu &
\Zar^b(Y_\bullet) \ar[r]^\mu &
\Zar(Y_\bullet)
}
\]
is commutative.
\end{lemma}

\begin{proof}
Obvious.
\end{proof}

\begin{lemma}\label{fpqc-Zariski.thm}
Let $\varphi:X\rightarrow Y$ be a qfpqc morphism of schemes.
Then 
\[
0\rightarrow \M\xrightarrow u 
\varphi_*\varphi^*\M \xrightarrow{\beta_1-\beta_2}\psi_*\psi^*\M
\]
is exact for $\M\in\Qch(Y)$,
where $u$ is the unit map for the adjoint pair $(\varphi^*,\varphi_*)$,
$p_i:X\times_Y X\rightarrow X$ is the $i$th projection, 
$\psi=\varphi p_1=\varphi p_2$, 
and $\beta_i$ is the composite
\[
\varphi_*\varphi^*\M 
\xrightarrow{u}
\varphi_*(p_i)_*(p_i)^*\varphi^*\M
\cong \psi_*\psi^*\M.
\]
\end{lemma}

\begin{proof}
This is equivalent to say that $\mu^*(\M)$ is already a sheaf in the
fpqc topology, which is equivalent to the qfpqc topology.
But this is trivial, as $\mu^*(\M)\cong \nu_*\nu^*\mu^*(\M)$.
\end{proof}

\paragraph
Let $\M\in\Qch_{\Fpqc}(X_\bullet)$,
and $(i,f:U\rightarrow X_i)\in\Fpqc(X_\bullet)$ with
$U=\Spec A$ affine.
Then $\M|_{(i,U)}\cong \tilde M$, where $M=\Gamma((i,U),\M)$, by
Lemma~\ref{qch.thm}.
For any covering of the form $\frak U:(i,\Spec B)\rightarrow (i,\Spec A)$ and
any $A$-module $M$, the \v Cech cohomology $\cc^i(\frak U,\tilde M)=0$ 
for $i>0$.
It follows that $\cc^i((i,U),\tilde M)=0$ for $i>0$.
Quite similarly to the proof of \cite[(2.12)]{Milne}, 
we can prove that 
$H^n((i,U),\M)=0$ for $U$ affine, $n>0$, and any quasi-coherent sheaf
on $\Fpqc(X_\bullet)$.

\paragraph
An $\O_{X_\bullet}$-module $\M$ is said to be {\em equivariant} if
the canonical map $\alpha_\phi:X_\phi^*\M_i\rightarrow \M_j$ is
an isomorphism for any morphism $\phi:i\rightarrow j$ in $I$, see
\cite[(4.14)]{ETI}.
The full subcategory of $\Mod(X_\bullet)$ consisting of 
equivariant $\O_{X_\bullet}$-modules is denoted by $\EM(X_\bullet)$.
Similarly, an equivariant module on 
$\Zar^b(X_\bullet)$ and 
$\Fpqc(X_\bullet)$ are defined, and the full subcategories of these, 
denoted by $\EM^b(X_\bullet)$ and
$\EM_{\Fpqc}(X_\bullet)$, respectively,
are defined.

\paragraph
An object $\M$ in $\Mod(X_\bullet)$ 
is said to be {\em locally quasi-coherent} if 
the restriction $\M_i$ is quasi-coherent for each $i\in I$.

\begin{lemma}
Let $\M\in \Mod(X_\bullet)$ \(resp.\  $\Mod^b(X_\bullet)$, 
$\Mod_{\Fpqc}(X_\bullet)$\).
Then $\M$ is
quasi-coherent if and only if it is locally quasi-coherent and equivariant.
\end{lemma}

\begin{proof}
The case $\M\in\Mod(X_\bullet)$ is \cite[(7.3)]{ETI}.
The other cases are proved using the discussion of 
\cite[(7.3)]{ETI}, Lemma~\ref{qch.thm}, and Lemma~\ref{mor-munu.thm}.
\end{proof}

\begin{lemma}\label{five.thm}
Assume that $X_\phi\in\Bbb C$ for $\phi\in\Mor(I)$.
If 
\[
\M_1\rightarrow \M_2\rightarrow \M_3\rightarrow \M_4\rightarrow \M_5
\]
is an exact sequence in $\Mod((\Bbb C/X_\bullet)_E)$ and
$\M_1$, $\M_2$, $\M_4$, and $\M_5$ are equivariant, then $\M_3$ is
equivariant.
In particular, $\EM((\Bbb C/X_\bullet)_E)$ is closed under extensions, 
kernels, and cokernels in $\Mod((\Bbb C/X_\bullet)_E)$, and hence
itself is an abelian category, and the
inclusion $\EM((\Bbb C/X_\bullet)_E)\hookrightarrow
\Mod((\Bbb C/X_\bullet)_E)$ is an exact functor.
\end{lemma}

\begin{proof}
For each morphism $\phi:i\rightarrow j$ of $I$, consider the 
commutative diagram
\[
\xymatrix{
X_\phi^*(\M_1)_i \ar[d]^{\alpha_\phi(\M_1)} \ar[r] &
X_\phi^*(\M_2)_i \ar[d]^{\alpha_\phi(\M_2)} \ar[r] &
X_\phi^*(\M_3)_i \ar[d]^{\alpha_\phi(\M_3)} \ar[r] &
X_\phi^*(\M_4)_i \ar[d]^{\alpha_\phi(\M_4)} \ar[r] &
X_\phi^*(\M_5)_i \ar[d]^{\alpha_\phi(\M_5)} \\
(\M_1)_j \ar[r] &
(\M_2)_j \ar[r] &
(\M_3)_j \ar[r] &
(\M_4)_j \ar[r] &
(\M_5)_j
}.
\]
The rows are exact by Lemma~\ref{restriction.thm} and
Lemma~\ref{inverse-image.thm}.
By assumption, $\alpha_\phi(\M_i)$ is an isomorphism for $i=1,2,4,5$.
By the five lemma, $\alpha_\phi(\M_3)$ is also an isomorphism.
Hence $\M_3$ is equivariant.
\end{proof}

\paragraph
Let $S$ be a scheme, and $G$ an $S$-group scheme.
Fix $\kappa$ sufficiently large so that $G\rightarrow S$ is a 
$\kappa$-morphism.
Let $X$ be a $G$-scheme.

\paragraph
We define 
$\Zar(G,X):=\Zar(B_G^M(X))$ and
$\Zar^+(G,X):=\Zar(B_G(X)_{(\Delta)\mon})$.
See for the notation, see \cite{ETI}.
Strictly speaking, the object set of 
$\Zar(X_\bullet)$ is slightly different in \cite{ETI}.
In our definition, an object of $\Zar(X_\bullet)$ is a pair 
$(i,h:U\rightarrow X_i)$ with $h$ an open immersion, while in \cite{ETI},
$h$ is required to be the inclusion map of an open subscheme.
But this difference will not cause any trouble.
The category  of $(G,\O_X)$-modules $\Mod(\Zar(G,X))$ is denoted by
$\Mod(G,X)$.
$\Mod(\Zar^+(G,X))$ is denoted by $\Mod^+(G,X)$.
Similar definitions are done in obvious ways.
For example, $\Fpqc(G,X)$ means $\Fpqc(B_G^M(X))$,
$\EM^{b,+}(G,X)$ means $\EM(\Zar^{b,+}(G,X))=\EM(\Zar^b(B_G(X)_{(\Delta)\mon}))$,
and $\Inv^{b}(G,X)$ means $\Inv(\Zar^{b}(G,X))
=\Inv(\Zar^{b}(B_G^M(X)))$, and so on.

A $(G,\O_X)$-module means an object of $\Mod(G,X)$.
An fpqc $(G,\O_X)$-module means an object of $\Mod_{\Fpqc}(G,X)$.

\paragraph\label{ten.par}
There are canonical restriction functors
\[
(?)_{(\Delta)\mon_{[0],[1],[2]}}:
\Mod^+(G,X)\rightarrow \Mod(G,X),
\]
\[
(?)_{(\Delta)\mon_{[0],[1],[2]}}:
\Mod_{\Fpqc}^+(G,X)\rightarrow \Mod_{\Fpqc}(G,X),
\]
and so on.
These functors induce equivalences
\[
\EM^+(G,X)\rightarrow \EM(G,X),
\]
\[
\Qch^+(G,X)\rightarrow \Qch(G,X),
\]
\[
\EM_{\Fpqc}^+(G,X)\rightarrow \EM_{\Fpqc}(G,X)
\]
and so on.
Thus the six categories
\begin{multline*}
\Qch^+(G,X),\; \Qch(G,X),\;\Qch^{b,+}(G,X),
\\ 
\Qch^{b}(G,X),\;\Qch_{\Fpqc}^+(G,X),\; \Qch_{\Fpqc}(G,X)
\end{multline*}
are equivalent and identified in a natural way.

\paragraph
The category $\EM(G,X)$ and $\Qch(G,X)$ are identified with the category 
of 
$G$-linearized
$\O_X$-modules and that of $G$-linearized quasi-coherent $\O_X$-modules
defined in \cite{GIT}, respectively.
Letting $\Qch/S$ denote the stack of quasi-coherent sheaves in 
(Zariski site)
over $\Sch/S$ with the fpqc topology (see \cite[(4.23)]{Vistoli}),
$\Qch(G,X)$ is also equivalent to the category $(\Qch/S(X))^G$ of
$G$-equivariant objects in $\Qch/S(X)$ (\cite[(3.48)]{Vistoli}).

\paragraph Let $F$ be an $S$-group scheme, and $X$ an $S$-scheme on which
$F$ acts trivially.
For an $(F,\O_X)$-module $\M$, $\M^F$ denotes the 
descended $\O_X$-module $(?)_{-1}R_{(\Delta_M)}\M$, where we consider that 
$\M\in\Mod(F,X)$ \cite[(30.1), (30.2)]{ETI}.
We call $\M^F$ the $F$-invariance of $\M$.
A similar definition can be done for the fpqc topology.

Let $\rho:p^*\M\rightarrow p^*\M$ be the $F$-linearization, where 
$p=p_2=a:F\times X\rightarrow X$ is the (trivial) action, which equals the
second projection.
Then $\M^F$ is the kernel of $(1-\rho)u:\M\rightarrow p_*p^*\M$, where 
$u:\M\rightarrow p_*p^*\M$ is the unit of adjunction of the adjoint
pair $(p^*,p_*)$.
We say that $\M$ is $F$-trivial if $\M^F=\M$.
If $p$ is quasi-compact and $\M$ is quasi-coherent, then
it is easy to see that $\M^F$ is also quasi-coherent.

\begin{proposition}\label{basic.thm}
Let $\varphi:X\rightarrow Y$ be a principal $G$-bundle.
Then the inverse image $\varphi^*: \Qch(Y)\rightarrow \Qch(G,X)$ 
is an equivalence.
The functor $(?)^G\circ \varphi_*: \Qch(G,X)\rightarrow \Qch(Y)$ is 
its quasi-inverse.
The unit of adjunction $u:\Id\rightarrow (?)^G\varphi_*\varphi^*$ is 
the map induced by the unit of adjunction
$u':\Id \rightarrow \varphi_*\varphi^*$ of the adjoint pair $(\varphi^*
,\varphi_*)$.
That is, $u$ is the unique map such that $u'=\iota u$, where 
$\iota:(?)^G\hookrightarrow\Id$ is the inclusion.
The counit of adjunction 
$\varepsilon:\varphi^*(?)^G\varphi_*\rightarrow \Id$ 
is the composite
\[
\varphi^*(?)^G\varphi_*\xrightarrow{\iota} \varphi^*\varphi_*
\xrightarrow{\varepsilon'}\Id,
\]
where $\varepsilon'$ is the counit of the adjunction of $(\varphi^*,\varphi_*)$.
\end{proposition}

\begin{proof}
$(\varphi^*,(?)^G\varphi_*)$ is an adjoint pair.
Indeed, 
the composite
\[
\varphi^*\xrightarrow{\varphi^* u}
\varphi^*(?)^G\varphi_*\varphi^*
\xrightarrow{\varepsilon\varphi^*}
\varphi^*
\]
is the identity almost by definition and the equality $(\varepsilon'\varphi^*)
\circ(\varphi^*u')=\id$.
Similarly, the composite
\[
(?)^G\varphi_*\xrightarrow{u(?)^G\varphi_*}
(?)^G\varphi_*\varphi^*(?)^G\varphi_*
\xrightarrow{(?)^G\varphi_*\varepsilon}
(?)^G\varphi_*
\]
is the identity, since $\iota$ is a monomorphism, and
the diagram
\[
\xymatrix{
(?)^G\varphi_* \ar[r]^-u \ar[d]^\iota &
(?)^G\varphi_*\varphi^*(?)^G\varphi_* \ar[d]^\iota \ar[r]^-\varepsilon &
(?)^G\varphi_* \ar[d]^\iota \\
\varphi_* \ar[r]^-u \ar[dr]^u 
\ar `d[ddr] `r[rr]^\id [rr]
& 
(?)^G\varphi_*\varphi^*\varphi_* \ar[ur]^\varepsilon  \ar[d]^\iota &
\varphi_* \\
 & \varphi_*\varphi^*\varphi_* \ar[ur]^\varepsilon & \\
 &  &
}
\]
is commutative.
As $(\varphi^*,(?)^G\varphi_*)$ is an adjoint pair and 
$\varphi^*$ is an equivalence by \cite[(4.23), (4.46)]{Vistoli},
we are done.
\end{proof}

Similarly, considering the stack of the all modules $\Mod_{\Fpqc,\kappa}/S$
(it is a stack, as in \cite[(4.11)]{Vistoli})
over $(\Sch/S)_\kappa$ with the fpqc topology, 
we have a similar result.

\begin{lemma}\label{basic2.thm}
Let $\varphi:X\rightarrow Y$ be a principal $G$-bundle.
Then the inverse image $\varphi^*: \Mod_{\Fpqc}(Y)\rightarrow \EM_{\Fpqc}(G,X)$ 
is an equivalence.
The functor $(?)^G\circ \varphi_*$ is its quasi-inverse.
The unit of adjunction $u:\Id \rightarrow (?)^G\varphi_*\varphi^*$ is the
map induced by the usual map $\Id \rightarrow\varphi_*\varphi^*$.
The counit of adjunction $\varepsilon:\varphi^*(?)^G\varphi_*\rightarrow
\Id$ is the composite
\[
\varphi^*(?)^G\varphi_*
\rightarrow
\varphi^*\varphi_*
\rightarrow
\Id.
\]
Under this equivalence, $\Qch_{\Fpqc}(Y)$ corresponds to $\Qch_{\Fpqc}(G,X)$.
\end{lemma}

\begin{proof}
We only need to prove the last assertion.
By definition, the quasi-coherence is local.
So $\varphi^*\N$ is locally quasi-coherent for $\N\in\Qch_{\Fpqc}(Y)$.
The equivariance of $\varphi^*\N$ is trivial, so $\varphi^*\N\in
\Qch_{\Fpqc}(G,X)$.
Conversely, if $\varphi^*\N$ is quasi-coherent, then $\N$ is quasi-coherent
simply by the local nature of quasi-coherence.
\end{proof}

\section{Enriched Grothendieck's descent}
\label{e-gro.sec}

\paragraph Let $S$ be a scheme, and $G$ an $S$-group scheme.

\begin{lemma}\label{universal-gq.thm}
  A principal $G$-bundle $\varphi:X\rightarrow Y$ 
  is a universal geometric quotient
  in the sense of \cite{GIT}.
  In particular, it is a categorical quotient, and hence is uniquely
  determined only by the $G$-scheme $X$.
\end{lemma}

\begin{proof}
  As a base change of a principal $G$-bundle is again a principal $G$-bundle,
  it is enough to show that a principal $G$-bundle is a geometric quotient.

  Let $\varphi:X\rightarrow Y$ be a principal $G$-bundle.
  By Proposition~\ref{basic.thm}, $u:\O_Y\rightarrow (\varphi_*\O_X)^G$ is
  an isomorphism.
  By Lemma~\ref{qfpqc-basic.thm}, {\bf 4}, $\varphi$ is 
  submersive.
  By assumption, $\varphi$ is $G$-invariant.
The map
  $\Phi:G\times X\rightarrow X\times_Y X$ is surjective,
since it is an isomorphism by Lemma~\ref{principal-mumford.thm}.
\end{proof}

\paragraph
From now on, until the end of this paper, 
let $f: G\rightarrow H$ be a qfpqc homomorphism
of $S$-group schemes, and $N=\Ker f$.
Note that $N$ is a normal subgroup scheme of $G$.

\paragraph\label{f-is-principal.par}
Note that 
$\Phi:N\times G\rightarrow G\times_H G$ given by $\Phi(n,g)=(ng,g)$ 
is an isomorphism, since $\Psi:G\times_H G\rightarrow N\times G$ given
by $\Psi(g,g_1)=(gg_1^{-1},g_1)$ is its inverse.
So $f$ is a $G$-enriched principal $N$-bundle by
Lemma~\ref{principal-mumford.thm}.

\begin{lemma}
  Let $X$ be a $G$-scheme on which $N$ acts trivially.
  Then there is a unique action $a'_X:H\times X\rightarrow X$ such that
  the diagram
  \begin{equation}\label{commutative.eq}
    \xymatrix{
      G\times X \ar[r]^-{a_X} \ar[d]^{f\times 1_X} & X \ar[d]^{1_X} \\
      H\times X \ar[r]^-{a'_X} & X
    }
  \end{equation}
  is commutative, where $a_X:G\times X\rightarrow X$ is the given action.
\end{lemma}

\begin{proof}
  Let $X'$ be the $S$-scheme $X$ with the trivial $G$-action.
  Then being a base change of $f$, $f\times 1_X: G\times X'\rightarrow
  H\times X'$ is a geometric quotient under the action of $N$ by
  Lemma~\ref{universal-gq.thm}.
  In particular, it is a categorical quotient under the action of $N$
by \cite[Proposition~0.1]{GIT}.
  As $a_X:G\times X'\rightarrow X$ is an $N$-invariant morphism, 
  there is a unique morphism $a'_X: H\times X\rightarrow X$ such that
  the diagram (\ref{commutative.eq}) is commutative.

  We compare the two maps $a'_X\circ (\mu_H\times 1_X)$ and $a'_X\circ
  (1_H\times a'_X)$ from $H\times H\times X$ to $X$.
  When we compose $f\times f\times 1_X : G\times G\times X
  \rightarrow H\times H\times X$ to the right, they agree by
  the commutativity of (\ref{commutative.eq}) and the facts that
  $f$ is a homomorphism, and that $a_X$ is an action.
  By Lemma~\ref{fpqc-equal.thm}, the two maps agree.
  It is easy to see that
  \[
  X\cong S\times X\xrightarrow{u_H\times 1_X} H\times X\xrightarrow{a_X'} X
  \]
  is the identity, where $u_H:S\rightarrow H$ is the unit element.
  Thus $a_X'$ is an action of $H$ on $X$.
\end{proof}

\paragraph
Conversely, if $a'_X:H\times X\rightarrow X
$ is a given $H$-action on $X$, then
defining $a_X=a_X'(f\times 1_X)$ (that is, in a unique way so that
the diagram (\ref{commutative.eq}) is commutative), 
$a_X:G\times X\rightarrow X$ is an action, and $N$ acts trivially on $X$.
From now on, we identify an $H$-scheme and a $G$-scheme on which 
$N$ acts trivially.

\begin{lemma}
Let $\varphi:X\rightarrow Y$ be a $G$-morphism which is also a
geometric quotient under the action of $N$.
Let $U$ be a $G$-stable open subset of $X$.
Then 
\begin{enumerate}
\item[\bf 1] $U=\varphi^{-1}(\varphi(U))$.
\item[\bf 2] $\varphi(U)$ is a $G$-stable open subset of $Y$.
\item[\bf 3] $\varphi|_U: U\rightarrow \varphi(U)$ is a base change of
$\varphi$ by a $G$-stable open immersion.
In particular, it is a $G$-morphism which is a geometric quotient under
the action of $N$.
If, moreover, $\varphi$ is a principal $N$-bundle (resp.\ trivial $N$-bundle),
then so is $\varphi|_U:U\rightarrow \varphi(U)$.
\end{enumerate}
\end{lemma}

\begin{proof}
Set $V:=\varphi(U)$.
{\bf 1} We may assume that $G=N$.
If $U\neq \varphi^{-1}(\varphi(U))$, then
there is a geometric point $\xi\rightarrow S$ such that
$\Phi:G(\xi)\times X(\xi)\rightarrow X(\xi)\times_{Y(\xi)}X(\xi)$
is surjective, and $U(\xi)\neq \varphi^{-1}(\varphi(U(\xi)))$.
Then there exists some $x\in U(\xi)$ and $y\in\varphi^{-1}(\varphi(U(\xi)))
\setminus U(\xi)$ such that $\varphi(x)=\varphi(y)$.
There exists some $g\in G(\xi)$ such that $\Phi(g,x)=(gx,x)=(y,x)$.
This contradicts the assumption that $U$ is $G$-stable.

{\bf 2} As $\varphi$ is submersive and $\varphi^{-1}(V)=U$ is open, 
$V$ is open.
As $\varphi\circ a_X=a_Y\circ(1_G\times\varphi)$ as morphisms from
$G\times X$ to $Y$, 
\[
G\cdot V=a_Y(G\times V)=a_Y(1_G\times\varphi)(G\times U)
=\varphi a_X(G\times U)=\varphi(U)=V.
\]
This shows that $V$ is $G$-stable.

{\bf 3} follows from {\bf 1} and {\bf 2}.
\end{proof}

\paragraph\label{last-setting.par}
Until the end of this section, 
we fix a regular cardinal $\kappa$ sufficiently large so that
the structure morphism $G\rightarrow S$ is a $\kappa$-morphism (such a 
$\kappa$ exists).

Let $X$ be an $H$-scheme.
We want to show that the category $\Qch(H,X)$ of 
quasi-coherent $(H,\O_X)$-modules and
the category $\Qch_N(G,X)$ of quasi-coherent $(G,\O_X)$-modules on which 
$N$ acts trivially are equivalent.
Similarly, we want to show that $\EM_{\Fpqc}(H,X)$ is
equivalent to the category of $N$-trivial equivariant $(G,\O_X)$-modules
$\EM_{N,\Fpqc}(G,X)$.

For the purpose above, we need the notion of $G$-linearized $\O_X$-modules
by Mumford \cite{GIT} (in the fpqc topology).

Let $F$ be an $S$-group scheme, and
$Z$ be an $F$-scheme.
An 
$F$-linearized $\O_Z$-module is a pair $(\M,\rho)$ such that
$\M\in\Mod_{\Fpqc}(Z)$, 
and $\rho:a^*\M\rightarrow p_2^*\M$
is an isomorphism such that the diagram
\[
\xymatrix{
  (\mu\times 1)^* a^*\M \ar[r]^\rho \ar[d]_\cong &
  (\mu\times 1)^* p_2^*\M \ar[r]^-\cong & p_{23}^*p_2^*\M \\
  (1\times a)^*a^*\M \ar[r]^\rho & 
  (1\times a)^*p_2^*\M \ar[r]^-\cong & p_{23}^*a^*\M \ar[u]_\rho
}
\]
is commutative, where $\mu:F\times F\rightarrow F$ is the product,
$a:F\times Z\rightarrow Z$ is the action, $p_2:F\times Z\rightarrow 
Z$ is the second projection, and $p_{23}:F\times F\times Z
\rightarrow F\times Z$ is given by $p_{23}(g,g',z)=(g',z)$.
When we need not mention $\rho$, sometimes we say that $\M$ is
an $F$-linearized $\O_Z$-module.
$\rho$ is called the $F$-linearization of $\M$.
Let $\M$ and $\N$ be $F$-linearized $\O_Z$-modules.
We say that $\phi:\M\rightarrow\N$ is a map of $F$-linearized $\O_Z$-modules
if $\phi$ is $\O_Z$-linear, and the diagram
\[
\xymatrix{
  a^*\M \ar[r]^-{a^*\phi} \ar[d]^\rho & a^*\N \ar[d]^\rho \\
  p_2^*\M \ar[r]^-{p_2^*\phi} & p_2^*\N
}
\]
is commutative.
The category $\LM_{\Fpqc}(F,Z)$ of $F$-linearized $\O_Z$-modules 
(in the fpqc topology) is obtained.
For $\M\in\EM_{\Fpqc}(F,Z)$, when we define $\rho:a^*\M_0\rightarrow p_2^*\M_0$ 
to be the composite
\[
a^*\M_0\xrightarrow{\alpha_{\delta_0}} \M_1 
\xrightarrow{\alpha_{\delta_1}^{-1}}p_2^*\M_0,
\]
then $(\M_0,\rho)$ is an $F$-linearized $\O_Z$-module, and we get
an equivalence $\EM_{\Fpqc}(F,Z)\rightarrow \LM_{\Fpqc}(F,Z)$.
We identify an equivariant $(F,\O_Z)$-module and an $F$-linearized
$\O_Z$-module.

We say that $(\M,\rho)\in \LM_{\Fpqc}(F,Z)$ is quasi-coherent if
$\M$ is so.
The full subcategory of $\LM_{\Fpqc}(F,Z)$ consisting of quasi-coherent 
$F$-linearized $\O_Z$-modules is denoted by $\LQ_{\Fpqc}(F,Z)$.
The equivalence $\EM_{\Fpqc}(F,Z)\rightarrow 
\LM_{\Fpqc}(F,Z)$ induces an equivalence
$\Qch_{\Fpqc}(F,Z)\rightarrow\LQ_{\Fpqc}(F,Z)$.

\paragraph
Let $F$ be as above, and let $\varphi:X\rightarrow Y$ be an $F$-morphism.
For $\M\in\LM_{\Fpqc}(F,Y)$, $\M\in\EM_{\Fpqc}(F,Y)$. So 
$\varphi^*\M\in\EM_{\Fpqc}(F,X)$.
So $\varphi^*\M$ has a structure of $F$-linearized $\O_X$-module.
It is easy to check that the $F$-linearization is given by the composite
\[
a_X^*\varphi^*\M\cong (1_G\times\varphi)^*a_Y^*\M
\xrightarrow{\rho}
(1_G\times\varphi)^*(p_2)_Y^*\M
\cong
(p_2)_X^*\varphi^*\M.
\]

\paragraph
Assume that $F\rightarrow S$ is a $\kappa$-morphism.
For $\N\in \EM_{\Fpqc}(F,X)$,
we have $\varphi_*\N\in\EM_{\Fpqc}(F,Y)$ 
(this is proved similarly to \cite[(7.14)]{ETI}).
So $\varphi_*\N$ has a structure of an $F$-linearized $\O_Y$-module.
The $F$-linearization is given by
\[
a_Y^*\varphi_*\N \xrightarrow{\theta}
(1_G\times \varphi)_*a_X^*\N\xrightarrow{\rho}
(1_G\times \varphi)_*(p_2)_X^*\N\xrightarrow{\theta^{-1}}
(p_2)_Y^*\varphi_*\N,
\]
where $\theta$ is Lipman's theta \cite[(1.21)]{ETI},
which is an isomorphism (this is true because we are considering the
fpqc site, see Lemma~\ref{theta.thm}).

\paragraph \label{res.par}
Let $h:F\rightarrow F'$ be a homomorphism of $S$-group schemes.
Let $Z$ be an $F'$-scheme.
Then $Z$ is an $F$-scheme in an obvious way, and we get a map
$B_h^M(Z): B_F^M(Z)\rightarrow B_{F'}^M(Z)$ defined as
\[
\xymatrix{
  B_F^M(Z): &F\times F\times Z \ar[d]_{h\times h\times 1_Z}
  \ar@<.75em>[r]
  \ar[r]
  \ar@<-.75em>[r] &
  F\times Z \ar[d]^{h\times 1_Z} 
  \ar@<.5em>[r]
  \ar@<-.5em>[r] &
  Z \ar[d]^{1_Z} \\
  B_{F'}^M(Z): & F'\times F'\times Z
  \ar@<.75em>[r]
  \ar[r]
  \ar@<-.75em>[r] &
  F'\times Z
  \ar@<.5em>[r]
  \ar@<-.5em>[r] &
  Z 
}.
\]

For the definition of $B_F^M(Z)$, see \cite[Chapter~29]{ETI}.
This induces the pull-back $B_h^M(Z)^*:\EM(F',Z)\rightarrow \EM(F,Z)$
by \cite[(7.22)]{ETI}.
We denote this functor by $\res^{F'}_F$.
Similarly, $\res^{F'}_{F,\Fpqc}=B_h^M(Z)^*:\EM_{\Fpqc}(F',Z)\rightarrow
\EM_{\Fpqc}(F,Z)$ is defined.

\paragraph Corresponding to $\res^{F'}_F$, we have a functor
$r^{F'}_F:\LM(F',Z)\rightarrow \LM(F,Z)$.
$r^{F'}_F(\M,\rho)=(\M,r(\rho))$, where 
$r(\rho)$ is the composite
\[
a_F^*\M\cong (h\times 1_X)^* a_{F'}^*\M\xrightarrow {(h\times 1_X)^*\rho}
(h\times 1_X)^* (p_2)_{F'}^*\M\cong (p_2)_{F}^*\M.
\]
Similarly, $r^{F'}_{F,\Fpqc}:\LM_{\Fpqc}(F',Z)\rightarrow \LM_{\Fpqc}(F,Z)$ is
obtained.

\paragraph
Let $F$ be an $S$-group scheme, and 
$X$ an $F$-scheme on which $F$-acts trivially.
It is easy to see that $\M\in\Mod(G,X)$ (resp.\ $\M\in\Mod_{\Fpqc}(G,X)$) is 
$F$-trivial if and only if 
$\M\cong \res^{\{e\}}_F \M_0$ for some $\M_0\in\Mod(X)=\Mod(\{e\},X)$
(resp.\  $\M_0\in\Mod_{\Fpqc}(X)=\Mod_{\Fpqc}(\{e\},X)$), where 
$\{e\}$ is the trivial group scheme.

\paragraph 
Let $X$ be an $H$-scheme.
By restriction, $\M=\res^H_G\M$ is an equivariant 
$(G,\O_X)$-module for an equivariant $(H,\O_X)$-module $\M$.
Note that $\res_N^G\res_G^H \M\cong \res^H_N\M\cong 
\res_N^{\{e\}}\res_{\{e\}}^H\M$ is $N$-trivial.
Let $\EM_N(G,X)$ (resp.\ $\Qch_N(G,X)$) 
denote the category of $N$-trivial equivariant (resp.\ quasi-coherent)
$(G,\O_X)$-modules.
Then we have a faithful exact functor 
$\res^H_G:\EM(H,X)\rightarrow
\EM_N(G,X)$.
The category of $N$-trivial $G$-linearized $\O_X$-modules
(resp.\ quasi-coherent $\O_X$-modules) is denoted by
$\LM_N(G,X)$ (resp.\ $\LQ_N(G,X)$).
Thus $\EM_N(G,X)$ is equivalent to $\LM_N(G,X)$ and
$\Qch_N(G,X)$ is equivalent to $\LQ_N(G,X)$.
Similar definitions can be done for the fpqc site.

\paragraph
Let $X$ be an $H$-scheme.
Let $(\M,\rho)\in\LM_N(G,X)$.
Thus $\M$ is an $\O_X$-module, and 
$\rho:a_G^*\M\rightarrow (p_2)_G^*\M$ is a $G$-linearization,
where $a_G:G\times X\rightarrow X$ and $(p_2)_G:G\times X\rightarrow X$
are the action and the second projection, respectively.

Note that $N$ acts on $G\times X$ by 
$\alpha:=(\mu\times 1_X)\circ(\iota\times 1\times 1):N\times G\times X
\rightarrow G\times X$, where $\iota:N\rightarrow G$ is the inclusion.
That is, $n(g,x)=(ng,x)$.
We define $q:N\times G\times X\rightarrow G\times X$ by 
$q(n,g,x)=(g,x)$.
Note also that
$f\times 1_X:G\times X\rightarrow H\times X$ is a principal 
$N$-bundle.
Thus $a_G^*\M\cong (f\times 1_X)^*a_H^*\M$ and
$(p_2)_G^*\M\cong (f\times 1_X)^*(p_2)_H^*\M$ have structures of 
$(N,\O_{G\times X})$-modules.
The $N$-linearization $\eta$ of $a_G^*\M$ is given by
\[
\alpha^*a_G^*\M\cong P^*\M \cong q^*a_G^*\M,
\]
where $P(n,g,x)=gx$.
The $N$-linearization $\zeta$ of $(p_2)_G^*\M$ is given similarly.

\begin{lemma}\label{crutial.thm}
  Let the notation be as above.
  Then $\rho:a^*\M\rightarrow p_2^*\M$ is a map of $N$-linearized 
  $\O_{G\times X}$-modules, where $a=a_G$ and $p_2=(p_2)_G$.
  A similar result holds for the fpqc site.
\end{lemma}

\begin{proof}
  Note that in the category of $\O_{G\times G\times X}$-modules 
  $\Mod(G\times G\times X)$, 
  \[
  \xymatrix{
    (\mu\times 1)^* a^*\M \ar[r]^\rho \ar[d]_\cong &
    (\mu\times 1)^* p_2^*\M \ar[r]^-\cong & p_{23}^*p_2^*\M \\
    (1\times a)^*a^*\M \ar[r]^\rho & 
    (1\times a)^*p_2^*\M \ar[r]^-\cong & p_{23}^*a^*\M \ar[u]_\rho
  }
  \]
  is commutative.
  Pulling back this diagram by $\iota\times 1_G\times 1_X: N\times G
  \times X\rightarrow G\times G\times X$, we get a commutative diagram
  \begin{landscape}
    \[
    \xymatrix{
      ~ & \ar `l[ld] `d[ddddrr] `[dddrr]^\zeta [dddrr]
      \alpha^*a^*\M \ar[d]^\cong \ar[r]^\rho 
      &
      \alpha^*p_2^*\M \ar[r]^\zeta \ar[d]^\cong &
      q^*p_2^*\M \ar[d]^\cong & \\
      & 
      (\iota\times 1\times 1)^*(\mu\times 1)^*a^*\M \ar[r]^\rho \ar[d]^\cong &
      (\iota\times 1\times 1)^*(\mu\times 1)^*p_2^*\M \ar[r]^\cong &
      (\iota\times 1\times 1)^*p_{23}^*p_2^*\M & \\
      & (\iota\times 1\times 1)^*(1\times a)^*a^*\M \ar[d]^\cong \ar[r]^\rho &
      (\iota\times 1\times 1)^*(1\times a)^*p_2^*\M \ar[d]^\cong \ar[r]^\cong &
      (\iota\times 1\times 1)^*p_{23}^*a^* \M\ar[d]^\cong \ar[u]_\rho & \\
      & (1\times a)^*(p_2)_N^*\M \ar[r]^{\id} & 
      (1\times a)^*(p_2)_N^*\M &
      q^*a^*\M 
      \ar `r[ruuu] `[uuu]_\rho [uuu]
      & ~\\
      ~ & ~ & ~& ~ & ~
    }.
    \]
  \end{landscape}
  This shows that $\rho:a^*\M\rightarrow p_2^*\M$ is a map of 
  $N$-linearized $\O_{G\times X}$-modules.
\end{proof}

\begin{lemma}\label{NGH.thm}
$\res^H_{G}:\Qch(H,X)\rightarrow \Qch_{N}(G,X)$ 
and
$\res^H_{G,\Fpqc}:\EM_{\Fpqc}(H,X)\rightarrow \EM_{N,\Fpqc}(G,X)$ 
are equivalences.
\end{lemma}

\begin{proof}
The proof is essentially the same, and we prove the assertion only for
the Zariski site.

  It suffices to show that $r^H_{G}:\LQ(H,X)\rightarrow 
\LQ_{N}(G,X)$ is an
  equivalence.

Let $(\M,\rho)\in\LQ_N(G,X)$.
By Lemma~\ref{crutial.thm} and 
  Proposition~\ref{basic.thm} (for the fpqc topology, use 
  Lemma~\ref{basic2.thm}), there exists a unique $\bar\rho:
  (a_H)^*\M\rightarrow (p_2)_H^*\M$ such that the composite
  \[
  a^*\M\cong (f\times 1)^*(a_H)^*\M
  \xrightarrow{\bar\rho}
  (f\times 1)^*(p_2)_H^*\M
  \cong p_2^*\M
  \]
  is $\rho$.
  By descent, it is easy to verify that $\bar\rho$ is an $H$-linearization
  of $\M$.
  As $(f\times 1)^*$ is a faithful functor, we have that
  $Q: \LQ_N(G,X)\rightarrow \LQ(H,X)$ given by 
  $Q(\M,\rho)=(\M,\bar\rho)$ is a functor.
  It is easy to see that $Q$ is a quasi-inverse of $r^H_G$.
\end{proof}

\begin{lemma}\label{complicated.thm}
  Let $F$ be an $S$-group scheme, and $M$ its normal subgroup scheme.
  Let $X$ be an $F$-scheme on which $M$ acts trivially.
Let $F\rightarrow S$ be a $\kappa$-morphism.
Let $\M\in \Mod_{\Fpqc}(F,X)$.
  Then $\M^M$ is the largest $M$-trivial 
  $(F,\O_X)$-submodule of $\M$.
  $(?)^M: \LM_{\Fpqc}(F,X)\rightarrow \LM_{M,\Fpqc}(F,X)$ 
  is a left exact functor.
\end{lemma}

\begin{proof}
By Lemma~\ref{inverse-image.thm}, $p_2^*, a^*:\Mod_{\Fpqc}(X)\rightarrow 
\Mod_{\Fpqc}(F\times X)$ are exact functors, and hence 
$a^*\M^M$ is a submodule of $a^*\M$, and $p_2^*\M^M$ is a submodule of
$p_2^\M$ (this discussion without the flatness of $p_2$ is possible
because we are working on the fppf site).

  We first prove that the $F$-linearization $\rho:a^*\M\rightarrow p_2^*\M$ 
  maps $a^*\M^M$ into $p_2^*\M^M$,
  where $a:F\times X\rightarrow X$ and $p_2:F\times X\rightarrow X$ are
  the action and the second projection, respectively.

  It suffices to show that the composite
  \[
  a^*\M^M \rightarrow a^*\M \xrightarrow{\rho} p_2^*\M
  \xrightarrow{u} p_2^*p_*p^*\M \xrightarrow{1-\rho} p_2^*p_*p^*\M
  \]
  is zero, where 
$p:M\times X\rightarrow X$ is the second projection, and
$u$ is the unit of adjunction.
  The diagram
  \[
  \xymatrix{
    p_2^*\M \ar[r]^u \ar[d]^u & p_2^*p_*p^*\M \ar[r]^{1-\rho} \ar[d]^\theta 
    \ar[d]^\theta & p_2^*p_*p^*\M \ar[d]^\theta \\
    (\mu\times1)_*(\mu\times1)^*p_2^*\M \ar[r]^-d &
    (\mu\times 1)_*p_{23}^*p^*\M \ar[r]^{1-\rho} &
    (\mu\times 1)_* p_{23}^*p^*\M
  }
  \]
  is commutative by the naturality of $\theta$ and \cite[(1.26)]{ETI}, 
  where $p:M\times X\rightarrow X$ is the trivial action,
  $\mu:F\times M\rightarrow F$ is the product, $p_{23}:F\times M\times X
  \rightarrow M\times X$ is the projection, $\theta$ is Lipman's theta
  \cite[(1.21)]{ETI} with respect to the cartesian square
  \[
  \xymatrix{
    F\times M\times X \ar[r]^-{p_{23}} \ar[d]^{\mu\times 1} & 
    M\times X \ar[d]^p \\
    F\times X \ar[r]^{p_2} & X
  },
  \]
  and $d$ is the canonical isomorphism 
  \cite[Chpaters 1, 2]{ETI}.
  As $\theta:p_2^*p_*p^*\M\rightarrow (\mu\times1)_*p_{23}^*p^*\M$ is
  an isomorphism by Lemma~\ref{theta.thm} 
(this is also because we treat the fpqc site), 
  through the adjoint isomorphism
  \[
  \Hom_{\O_{F\times X}}(a^*\M^M,(\mu\times 1)_*p_{23}^*p^*\M)
  \cong
  \Hom_{\O_{F\times N\times X}}((\mu\times 1)^*a^*\M^N,
  p_{23}^*p^*\M),
  \]
  it suffices to show that the composite
  \[
  (\mu\times 1)^*a^*\M^M
  \rightarrow
  (\mu\times 1)^*a^*\M
  \xrightarrow{\rho}
  (\mu\times 1)^*p_2^*\M
  \xrightarrow{d}
  p_{23}^*p^*\M
  \xrightarrow{\rho^{-1}-1}
  p_{23}^*p^*\M
  \]
  is zero.

  By the definition of the $F$-linearization, we have a commutative
  diagram of $\O_{M\times G\times X}$-modules
  \begin{equation}\label{tricky.eq}
    \xymatrix{
      (\mu'\times1)^*a^*\M^M
      \ar[r]^d \ar[d]^\rho &
      (1\times a)^*p^*\M^M \ar[r]^{\rho=1} &
      (1\times a)^*p^*\M^M \ar[d]^-d \\
      (\mu'\times 1)^*p_2^*\M \ar[r]^d &
      p_{23}^*p_2^*\M &
      p_{23}^*a^*\M^M \ar[l]_\rho
    }
  \end{equation}
  is commutative, where $\mu':M\times F\rightarrow F$ is the product.
  Let $\Theta: F\times M\times X\rightarrow M\times F\times X$ be
  the isomorphism given by $\Theta(f,m,x)=(fmf^{-1},f,x)$.
  Applying $\Theta^*$ to (\ref{tricky.eq}), we get the commutative diagram
  \[
  \xymatrix{
    (\mu\times 1)^*a^*\M^M \ar[r]^\rho \ar[d]^d &
    (\mu\times 1)^*p_2^*\M \ar[d]^d \\
    \Theta^*(\mu'\times 1)^*a^*\M^M \ar[r]^\rho \ar[d]^d &
    \Theta^*(\mu'\times 1)^*p_2^*\M \ar[d]^d \\
    \Theta^*(1\times a)^*p^*\M^M \ar[d]^{\rho=1} &
    \Theta^*p_{23}^*p_2^*\M \\
    \Theta^*(1\times a)^*p^*\M^M \ar[r]^d &
    \Theta^*p_{23}^*a^* \M \ar[u]_{\rho}
  }
  \]
  because $(\mu'\times 1)\Theta=\mu\times 1$.
  So the diagram
  \begin{equation}\label{tricky2.eq}
    \xymatrix{
      (\mu\times 1)^*a^*\M^M \ar[r]^\rho \ar[d]^d &
      (\mu\times 1)^*p_2^*\M \ar[d]^d \\
      (1_F\times p)^*a^*\M^M\ar[r]^\rho &
      (1_F\times p)^*p_2^*\M
    }
  \end{equation}
  is commutative (note that $p_{23}\Theta=1_F\times p$).
  On the other hand, by the definition of $F$-linearization, the diagram
  \begin{equation}\label{tricky3.eq}
    \xymatrix{
      (\mu\times 1)^*a^*\M^M \ar[r]^\rho \ar[d]^d &
      (\mu\times 1)^*p_2^*\M \ar[r]^d &
      p_{23}^*p^* \M \ar[d]^{\rho^{-1}} \\
      (1\times p)^*a^*\M^M \ar[r]^\rho &
      (1\times p)^*p_2^*\M \ar[r]^d &
      p_{23}^*p^*\M
    }
  \end{equation}
  is commutative.
  Combining the commutativity of (\ref{tricky2.eq}) and (\ref{tricky3.eq}),
  we have that the composite
  \[
  (\mu\times 1)^*a^*\M^M
  \xrightarrow{\rho}
  (\mu\times 1)^*p_2^*\M
  \xrightarrow{d}
  p_{23}^*p^*\M
  \xrightarrow{\rho^{-1}-1}
  p_{23}^*p^*\M
  \]
  is zero, as desired.

  Next, we want to prove that $\rho^{-1}:p_2^*\M\rightarrow a^*\M$ maps
  $p_2^*\M^M$ into $a^*\M^M$.
  This is proved similarly, using the commutativity of
  \[
  \xymatrix{
    (\mu\times 1)^*a^*\M \ar[d]^d & 
    (\mu\times 1)^*p_2^*\M^M \ar[d]^d \ar[l]_{\rho^{-1}} \\
    (1_F \times p)^*a^*\M &
    (1_F\times p)^*p_2^*\M^M \ar[l]_{\rho^{-1}}
  }
  \]
  and the commutativity of 
  \[
  \xymatrix{
    \Theta^*(\mu'\times 1)^*a^*\M &
    \Theta^*(\mu'\times 1)^*p_2^*\M \ar[l]_{\rho^{-1}} \ar[d]^d \\
    \Theta^*(1\times a)^*p^*\M \ar[u]^d &
    \Theta^*p_{23}^*p_2^*\M^N \ar[d]^{\rho^{-1}} \\
    \Theta^*(1\times a)^*p^*\M \ar[u]^{\rho^{-1}} &
    \Theta^*p_{23}^*a^*\M \ar[l]_d
  }.
  \]

  Thus $\rho:a^*\M^M\rightarrow p_2^*\M^M$ is an isomorphism, and
  thus $\M^M$ is an $F$-linearized $\O_X$-submodule of $\M$.
  It is $M$-trivial almost by the definition of $\M^M$.
  The other assertions are easy.
\end{proof}

\begin{lemma}\label{j-compatibility.thm}
Let $I$ be a small category, and $f_\bullet:X_\bullet\rightarrow Y_\bullet$ 
be a morphism of $I\op$-diagrams of $S$-schemes.
Let $j_{X_\bullet}=\mu\nu
:\Fpqc(X_\bullet)\rightarrow\Zar(X_\bullet)$ be the canonical
morphism of ringed sites.
Define $j_{Y_\bullet}$ in a similar way.
Then
\begin{enumerate}
\item[\bf 1] Lipman's theta 
$\theta:j^*_{X_\bullet}(?)_J\rightarrow (?)_Jj^*_{X_\bullet}$ is an isomorphism
for any subcategory $J$ of $I$.
\item[\bf 2] 
The conjugate $\xi: (j_{X_\bullet})_*R_J\rightarrow R_J(j_{X_\bullet|_J
})_*$ is also an isomorphism.
\item[\bf 3]
The Lipman's theta $\theta:f_\bullet^*(j_{Y_\bullet})_*\M
\rightarrow (j_{X_\bullet})_*f_\bullet^*\M$ 
is an isomorphism for locally quasi-coherent $\M$.
\end{enumerate}
\end{lemma}

\begin{proof}
{\bf 1} is proved similarly to \cite[(6.25)]{ETI}.
{\bf 2} follows from {\bf 1}.
To prove {\bf 3}, 
we may assume that $f:X\rightarrow Y$ is a map of single schemes.
It is easy to check the assertion for the case that $f$ is an open immersion.
So we may assume that both $X=\Spec B$ and $Y=\Spec A$ are affine, and 
$\M=\tilde M$.
Then it is easy to see that $\theta$ is the identity map $(B\otimes_A M)\,
\tilde{} \to (B\otimes_A M)\,\tilde{}$.
\end{proof}

\begin{theorem}\label{principal-main.thm}
  Let $\varphi:X\rightarrow Y$ be a $G$-enriched principal 
  $N$-bundle.
  Then $(?)^N\varphi_*:\EM_{\Fpqc}(G,X)\rightarrow \EM_{\Fpqc}(H,Y)$ 
  is an equivalence,
  and $\varphi^*:\EM_{\Fpqc}(H,Y)\rightarrow\EM_{\Fpqc}(G,X)$ 
  is its quasi-inverse.
  Under the equivalence, $\Qch_{\Fpqc}(H,Y)$ corresponds to 
  $\Qch_{\Fpqc}(G,X)$.
  The unit of adjunction $u:\Id\rightarrow (?)^N\varphi_*\varphi^*$ and
  the counit of adjunction $\varphi^*(?)^N\varphi_*\rightarrow\Id$ are
  given as in Proposition~\ref{basic.thm}.
\end{theorem}

\begin{proof}
  $\varphi_*$ is a functor from $\EM(G,X)$ to $\EM(G,Y)$
(again, we use the commutativity of Lipman's theta.
This is known to be true for quasi-coherent sheaves over small Zariski 
site only for the case that $G$ is flat and $f$ is quasi-compact 
quasi-separated, see \cite[(7.12), (7.14)]{ETI}).
  On the other hand, $(?)^N$ is a functor from $\EM(G,Y)$ to
  $\EM_N(G,Y)\cong \EM(H,Y)$ by Lemma~\ref{complicated.thm} and
  Lemma~\ref{NGH.thm}.
  So $(?)^N\varphi_*$ is a functor from $\EM(G,X)$ to $\EM(H,Y)$.
  Conversely,
  \[
  \EM(H,Y)\xrightarrow{\res}\EM_N(G,Y)\xrightarrow{\varphi^*}\EM(G,X)
  \]
  is a functor.

  These two functors are quasi-inverse each other,
  since $u:\Id\rightarrow (?)^N\varphi_*\varphi^*$ and 
  $\varepsilon:\varphi^*(?)^N\varphi_*\rightarrow\Id$ given in 
  Lemma~\ref{basic2.thm} (the group $N$ here is $G$ there) are
  $(G,\O_X)$-linear.
  
  The last assertion follows from the last assertion of Lemma~\ref{basic2.thm}.
\end{proof}

\begin{corollary}\label{descent-main-cor.thm}
Under the same assumption as in the theorem,
$(?)^N\varphi_*:\Qch(G,X)\rightarrow \Qch(H,Y)$ and
$\varphi^*: \Qch(H,Y)\rightarrow\Qch(G,X)$ are quasi-inverse each other.
The unit map 
and the counit map are given as in Proposition~\ref{basic.thm}.
\end{corollary}

\begin{proof}
It is clear that $\varphi^*: \Qch(H,Y)\rightarrow\Qch(G,X)$ is a functor.
Consider the composite of equivalences
\[
\Qch(G,X)\xrightarrow{j^*}\Qch_{\Fpqc}(G,X)\xrightarrow{(?)^N\varphi_*}
\Qch_{\Fpqc}(H,Y)\xrightarrow{j_*}\Qch(H,Y),
\]
where $j=\mu\nu$.
As in \cite[(30.3)]{ETI}, $(?)^N:\Mod(G,Y)\rightarrow\Mod(Y)$ is
$(?)_{[-1]}\circ R_{\Delta}$.
So by Lemma~\ref{j-compatibility.thm}, the composite is identified with
$(?)^N\varphi_*j_*j^*:\Qch(G,X)\cong \Qch(H,Y)$.
As $j_*j^*:\Qch(G,X)\rightarrow \Qch(G,X)$ is an autoequivalence,
$(?)^N\varphi_*:\Qch(G,X)\rightarrow \Qch(H,Y)$ is an equivalence.
The rest is easy.
\end{proof}

\begin{remark}
The statement of Corollary~\ref{descent-main-cor.thm} is 
independent of the choice of $\kappa$, as it is an assertion for
the Zariski topology.
\end{remark}

\section{Equivariant Picard groups and class groups}

\paragraph Let $f:G\rightarrow H$ be a qfpqc homomorphism of $S$-group
schemes, and $N=\Ker f$, as above.

\begin{corollary}\label{princ-pic-equiv.thm}
Let $\varphi:X\rightarrow Y$ be a $G$-enriched principal $N$-bundle.
Then $\varphi^*:\Inv(H,Y)\rightarrow\Inv(G,X)$ is an equivalence, and
$(?)^N\circ \varphi_*:\Inv(G,X)\rightarrow \Inv(H,Y)$ is its quasi-inverse.
Thus we have an isomorphism
\[
\varphi^{*}: [\M]\mapsto [\varphi^*\M]
\]
from $\Pic(H,Y)$ to $\Pic(G,X)$.
Its inverse is given by $[\N]\mapsto[(\varphi_*\N)^N]$.
\end{corollary}

\begin{proof}
This is immediate from
Corollary~\ref{descent-main-cor.thm} and Lemma~\ref{inv.thm}.
\end{proof}

\paragraph
Now assume that $G$ and $N$ are flat over $S$.
Being a principal $N$-bundle, $f$ is also flat, and hence is fpqc.
So $H$ is also flat over $S$.

\begin{corollary}\label{princ-ref-equiv.thm}
Let $\varphi:X\rightarrow Y$ be a $G$-enriched principal $N$-bundle
such that $X$ is locally Krull.
Then $\varphi^*:\Ref(H,Y)\rightarrow\Ref(G,X)$ is an equivalence, and
$(?)^N\circ \varphi_*:\Ref(G,X)\rightarrow \Ref(H,Y)$ is its quasi-inverse.
With this equivalence, $\Ref_n(H,Y)$ corresponds to $\Ref_n(G,X)$.
Thus $\Ref_1(H,Y)$ and $\Ref_1(G,X)$ are equivalent, and we have an 
isomorphism
\[
\varphi^{*}: [\M]\mapsto [\varphi^*\M]
\]
from $\Cl(H,Y)$ to $\Cl(G,X)$.
Its inverse is given by $[\N]\mapsto[(\varphi_*\N)^N]$.
\end{corollary}

\begin{proof}
Note that $\varphi$ is fpqc by Lemma~\ref{principal-descent.thm}, 
as $N\rightarrow S$ is fpqc.
Now the assertion follows from
Corollary~\ref{descent-main-cor.thm}
and \cite[(5.32)]{Hashimoto4}.
\end{proof}

\paragraph
Temporarily forget our settings on $G$, $H$, and $N$.

\begin{example}\label{equivariant-bundle.ex}
Let $N$ be an $S$-group scheme, and $H$ another $S$-group scheme
acting on $N$ as group automorphisms.
We say that $X$ is an $H$-equivariant $N$-scheme when $X$ is an $H$-scheme
$N$-scheme such that the action $a_X:N\times X\rightarrow X$ is an
$H$-morphism.
When we set $G:=N\rtimes H$, the semidirect product, then an 
$H$-equivariant $N$-scheme and a $G$-scheme is the same thing.
We define:
An {\em $H$-equivariant $N$-morphism} is a $G$-morphism.
An {\em $H$-equivariant $N$-invariant morphism} is a $G$-morphism
which is $N$-invariant.
An {\em $H$-equivariant principal $N$-bundle} is 
a $G$-enriched principal $N$-bundle.
\end{example}

Thus our results also apply to equivariant principal bundles.

\begin{example}
Let $k$ be a field, and $N_0$ a finite \'etale $k$-group scheme, and
$\varphi:X\rightarrow Y$ a principal $N_0$-bundle.
Let $k'$ be a finite Galois extension of $k$ such that $k'\otimes N_0$
is a constant finite group $N$.
That is, $N$ is a finite group and 
$k'\otimes k[N_0]\cong k'\otimes k[N]$ as $k'$-Hopf algebras.
We understand that $N$ also denotes the constant group scheme over $k$.
So $k'\otimes N_0\cong k'\otimes N$.
Note that the finite group $N$ is identified with the group of $k'$-valued
points of $N_0$, $N_0(k')=(\Sch/k)(\Spec k',N_0)$.

Let $H$ be the Galois group of $k'/k$.
$H$ acts on $N_0$ trivially, and it also acts on $k'\otimes k[N_0]
\cong k'\otimes k[N]$.
As an algebra automorphism preserves idempotents and $k[N]$ is the
$k$-subalgebra generated by the idempotents of $k'\otimes k[N]$, 
$H$ acts on $k[N]$, and so $H$ acts on $N$.
Thus the composite
\[
N\times \Spec k'\times X \cong N_0\times \Spec k'\times X \xrightarrow{\alpha}
\Spec k'\times X
\]
is an action of $N$ on $X'=\Spec k'\times X$, 
and the action is $H$-equivariant,
where $\alpha(n_0,w,x)=(w,n_0x)$.
Now it is easy to see that the base change
$\varphi':X'\rightarrow Y'$ of $\varphi$ by $\Spec k'\rightarrow \Spec k$
is an $H$-equivariant principal $N$-bundle, as a $k$-morphism.
So it is also a $G$-enriched principal $N$-bundle, where $G=N\rtimes H$.
Note that the diagram of equivalences is commutative up to natural
isomorphisms
\[
\xymatrix{
\Qch(G,X') & \Qch(N,X) \ar[l] \\
\Qch(H,Y') \ar[u]^{(\varphi')^*} & \Qch(Y) \ar[u]^{\varphi^*} \ar[l]
}.
\]
Thus $(\varphi')^*$ does almost the same thing as $\varphi^*$,
but $G$ and $H$ are constant groups, and no group scheme appears, while
$\varphi$ is a principal $N_0$-bundle, and $N_0$ need not be 
constant.
\end{example}

\end{document}